\documentclass[journal]{new-aiaa}
\usepackage[utf8]{inputenc}
\usepackage{textcomp}
\usepackage{graphicx}
\usepackage{placeins}
\usepackage{amsmath,bm,mathtools}
\usepackage[version=4]{mhchem}
\usepackage{siunitx}
\usepackage{longtable,tabularx,booktabs,array,multirow}
\usepackage{xcolor}
\usepackage{etoolbox}
\usepackage{tikz}
\usepackage{enumitem}
\usetikzlibrary{arrows.meta,positioning,fit,shapes.geometric}
\hypersetup{hidelinks}
\AtBeginEnvironment{matrix}{\renewcommand{\arraystretch}{1.0}}
\AtBeginEnvironment{pmatrix}{\renewcommand{\arraystretch}{1.0}}
\AtBeginEnvironment{bmatrix}{\renewcommand{\arraystretch}{1.0}}
\AtBeginEnvironment{Bmatrix}{\renewcommand{\arraystretch}{1.0}}
\AtBeginEnvironment{vmatrix}{\renewcommand{\arraystretch}{1.0}}
\AtBeginEnvironment{Vmatrix}{\renewcommand{\arraystretch}{1.0}}

\newcommand{\tr}{\operatorname{tr}}
\newcommand{\diag}{\operatorname{diag}}
\newcommand{\col}{\operatorname{col}}

\newcommand{\E}{\mathbb{E}}
\newcommand{\R}{\mathbb{R}}

\newcommand{\Phix}{\Phi^{x}}
\newcommand{\Phiu}{\Phi^{u}}

\newtheorem{theorem}{Theorem}

\newtheorem{corollary}{Corollary}
\newtheorem{remark}{Remark}
\newenvironment{proof}[1][Proof]{\par\noindent\textit{#1.}\ }{\hfill$\square$\par}

\title{Fast Stochastic Post-Fault Replanning for Launch Vehicles under Persistent Parametric Uncertainty}

\author[1]{Dai Shen\footnote{School of Astronautics. Email: shendaistars@buaa.edu.cn.}}
\author[1,2]{Chen Yang\footnote{Associate Professor, School of Astronautics. Email: cyang@buaa.edu.cn.}}
\author[1]{Ziqi Xu\footnote{School of Astronautics. Email: asxzq000617@163.com.}}
\author[1,3]{Lin Cheng\footnote{Associate Professor, School of Astronautics. Email: chenglin5580@buaa.edu.cn.}}
\author[4]{Hongbo Zhang\footnote{Professor, College of Aerospace Science and Engineering. Email: zhanghb1304@nudt.edu.cn.}}
\author[5]{Hongtao Zheng\footnote{Researcher. Email: 21583207@qq.com.}}
\author[1,2]{Shengping Gong\footnote{Professor, School of Astronautics. Corresponding author: gongsp@buaa.edu.cn.}}

\affil[1]{Beihang University, Beijing 102206, China}
\affil[2]{State Key Laboratory of High-Efficiency Reusable Aerospace Transportation Technology, Beijing 102206, China}
\affil[3]{Beijing Key Laboratory of Reusable Space Vehicle System Design, Beijing 102206, China}
\affil[4]{National University of Defense Technology, Changsha, Hunan 410072, China}
\affil[5]{Changsha Aerospace Technology Innovation Institute, Changsha, Hunan 410000, China}

\begin{document}
\maketitle

\begin{abstract}

Post-fault ascent replanning for launch vehicles requires rapid joint trajectory–guidance replanning under uncertainty to maintain closed-loop constraint satisfaction. In particular, propulsion performance cannot be promptly and fully identified after a fault, leaving persistent parametric uncertainty throughout the remaining flight. A fast stochastic ascent replanning method is developed to address this problem. To efficiently account for the cross-time correlation induced by persistent parametric uncertainty, a fast system level synthesis formulation is developed with a Riccati-based solution structure. The resulting feedback synthesis is complemented by a task-oriented terminal weighting method that shapes terminal covariance according to post-injection correction demand. These components are integrated with nominal trajectory replanning, chance-constraint tightening, and online feedback realization to establish a real-time joint stochastic trajectory–guidance replanning framework for launch vehicles. Nonlinear Monte Carlo simulations show that the proposed method effectively suppresses the influence of persistent parametric uncertainty and maintains stable convergence and high injection accuracy across parameter sweeps, with subsecond solution times. Compared with conventional ascent guidance, the empirical constraint-violation rate is substantially reduced.
\end{abstract}

\clearpage

\section*{Nomenclature}
\begin{longtable}{@{}p{0.29\textwidth}p{0.67\textwidth}@{}}
$A_k,B_k^\pm,D_k$ & FOH-discrete physical perturbation matrices \\
$\widetilde A_k,\widetilde B_k,\widetilde D_k$ & Augmented FOH matrices with energy-based nodes \\
$B_{q_{\mathrm e}},G_{q_{\mathrm e}}$ & RTN-impulse-to-MEE and Cartesian-state-to-MEE Jacobians \\
$\bm C_\alpha,C_{k,j}$ & Cholesky factor of $\bm M_\alpha$ and its $(k,j)$ block \\
$d_k,\bm d$ & Abstract additive disturbance and its full-horizon stack \\
$H_{\mathrm{orb}}$ & Stacked candidate-impulse map in phase-free MEE \\
$J_{\mathrm{corr},1}^\star,J_{\mathrm{corr},2}^\star$ & Minimum post-injection total impulse magnitude and correction energy \\
$\bm K$ & Causal augmented-state history-feedback operator \\
$L_b^{(i)},L_\eta^{(i)}$ & Persistent and interval-wise uncertainty-loading matrices \\
$L_{\mathrm{imp}},W_{\mathrm{imp}}$ & Length and width of the rectangular impact region \\
$\bm M_{\mathrm{phys}},\bm M_\alpha$ & Physical and regularized full-horizon disturbance covariances \\
$N,n_b$ & Number of discrete propagation steps and persistent-coordinate dimension \\
$P_0,P_{\chi,k},P_{\zeta_r}$ & Initial physical-state, augmented-state, and constraint-variable covariances \\
$Q_{\mathrm c},R_{\mathrm c}$ & Continuous-time state and control weights \\
$Q_{\mathrm f},\widetilde Q_{\mathrm f}$ & Terminal penalty matrix and augmented terminal weight \\
$Q_{\mathrm f,x}^{\mathrm{orb}},\lambda_{\mathrm f}$ & Task-oriented terminal matrix and its positive scale factor \\
$q_{\mathrm e}$ & Five-dimensional phase-free modified equinoctial element vector \\
$q_{\mathrm h},q_{\mathrm d}$ & Heat flux and dynamic pressure \\
$r,v$ & ECI position and velocity vectors \\
$u=[\phi\ \psi]^{\mathsf T}$ & Launch-inertial pitch and yaw thrust-direction command \\
$W_k,\bm W$ & Discrete stage and full-horizon quadratic weights \\
$x=[r^{\mathsf T}\ v^{\mathsf T}\ m]^{\mathsf T}$ & Seven-dimensional physical ascent state \\
$\alpha$ & Disturbance-covariance regularization parameter \\
$\beta_r,\pi_r,\kappa_r$ & Chance-constraint backoff, violation probability, and normal quantile \\
$\Delta v,\Delta V$ & Equivalent post-injection correction velocity and stacked impulses \\
$\varepsilon,\delta t_k$ & Specific mechanical energy and energy-node event-time deviation \\
$\mathcal F_{\mathrm N}$ & Standard normal cumulative distribution function \\
$\chi_k,\nu_k$ & Augmented perturbation state and FOH increment input \\
$\xi_b,\xi_k$ & Shared persistent and interval-wise independent standardized coordinates \\
$\bm\Phi^x,\bm\Phi^u$ & Full-horizon state and input system responses \\
$\Psi_{k,j}^x$ & State response factor on standardized random coordinates \\
$\widehat d_k$ & Reconstructed disturbance innovation \\
\end{longtable}
\addtocounter{table}{-1}

\section{Introduction}
\lettrine{N}{onfatal} propulsion faults in launch vehicles, such as partial thrust loss, do not necessarily terminate a launch mission if sufficient performance remains. However, achieving the nominal target orbit or a degraded rescue orbit requires rapid onboard trajectory replanning and guidance law updates \cite{Miao2022}. Compared with nominal ascent planning, post-fault replanning is significantly more challenging due to strict real-time requirements and pronounced propulsion uncertainties. Navigation systems accurately measure the vehicle's kinematic states, but thrust and specific impulse must be estimated indirectly from engine telemetry and vehicle dynamics. Constrained by fault transients and the limited remaining burn time, online identification cannot instantaneously and fully capture the actual performance degradation \cite{Wu2026Liquidpropellant,Yu2026ThrustDiag,Zhang2022ThrustDiag}. Consequently, residual errors in the reconstructed propulsion model persist throughout the remaining flight and act as persistent parametric uncertainty that accumulates over time \cite{Giordano2024Characterization}. Such uncertainty can induce significant trajectory deviations and constraint violations. Reliable post-fault replanning must therefore not only rapidly generate a feasible reference trajectory, but also guarantee closed-loop constraint satisfaction and orbital injection accuracy under residual propulsion uncertainties.

Most existing real-time post-fault replanning methods rely on deterministic convex optimization to recompute a nominal trajectory \cite{Miao2022,Miao2023Aux,Ma2024FaultGuidance}. Although these methods guarantee that the nominal solution satisfies path constraints, they cannot directly quantify the risk of constraint violations for the actual closed-loop trajectory under uncertainty. In practice, constraint satisfaction typically relies on empirically preset margins, which inevitably introduces excessive conservatism. Alternatively, stochastic trajectory optimization offers a more systematic approach by explicitly propagating state and control uncertainties, enforcing chance constraints to bound violation risks, and co-optimizing the nominal trajectory with closed-loop feedback \cite{Marmo2025Chance}.

In aerospace stochastic optimal control, linear covariance propagation is widely adopted due to its high computational efficiency and compatibility with convex optimization \cite{Geller2006}. Existing linear-covariance-based methods generally follow three representative computational strategies. The first parameterizes the causal feedback policy over the entire prediction horizon through a block-lower-triangular feedback map, and formulates covariance steering as a convex optimization problem \cite{Okamoto2018Cov,Ridderhof2022ChanceConstrained,Li2025Convex,Zhang2025Secondorder}. The second treats nodal covariance matrices as decision variables and convexifies the covariance dynamics through variable transformations and linear matrix inequality relaxations, thereby enabling the joint optimization of the nominal trajectory and covariance \cite{Benedikter2022Cov,Liu2025CovSteering,Kumagai2025Robust}. The third adopts an alternating optimization framework that decouples nominal trajectory planning and feedback synthesis into two smaller subproblems \cite{Messerer2021SIRO,Leeman2024,Shen2026IAC,Messerer2026RiccatiZORO}.  While the first two categories have been extensively investigated in aerospace applications, their real-time implementation remains bottlenecked by high dimensionality. Conversely, by avoiding large-scale joint optimization, the third category substantially reduces the computational burden and holds great promise for online post-fault replanning.

Despite the computational advances above, the unique disturbance mechanisms and terminal covariance requirements of post-fault launch vehicle ascent pose two major challenges to stochastic trajectory optimization. One challenge is fast feedback synthesis under persistent parametric uncertainty. Although residual propulsion model errors induced by incomplete identification significantly impact the flight trajectory, existing Riccati-based fast feedback methods typically assume white process noise \cite{Messerer2021SIRO,Messerer2026RiccatiZORO,Leeman2024}. 

Methods for non-white disturbances generally require complex optimization \cite{Knaup2023Efficient}. State augmentation restores the Markov property at the expense of requiring high-accuracy real-time thrust identification\cite{Ridderhof2022Entry}, but such identification is difficult within the limited time available after a fault.  Consequently, directly characterizing the closed-loop system response to persistent parametric uncertainty across the entire horizon while preserving a fast feedback structure remains highly challenging without high-precision online identification. Another challenge concerns the evaluation and optimization of the desired terminal covariance structure. Existing stochastic planning methods still rely heavily on empirical tuning to regulate terminal uncertainty, making it difficult to mechanistically determine which covariance structure benefits mission performance. An improper terminal covariance shape may hinder iterative convergence, induce overly aggressive feedback gains, or even cause closed-loop trajectory divergence. Therefore, eliminating this reliance on manual tuning to systematically obtain a well-behaved terminal covariance distribution in closed-loop design remains to be resolved.

In summary, existing deterministic replanning and guidance methods suffer from excessive constraint conservatism due to manually preset margins, whereas stochastic planning approaches remain bottlenecked by computational efficiency, the ability to handle persistent parametric uncertainty, and terminal covariance shaping strategies. To overcome these limitations, this work focuses on rapid joint stochastic trajectory--guidance replanning to simultaneously generate a feasible nominal trajectory while explicitly regulating closed-loop constraint violation risk and terminal injection dispersion under disturbances. The main contributions are threefold. First, a fast stochastic feedback synthesis method capable of handling persistent parametric uncertainty is proposed. Building on system level synthesis (SLS) \cite{Anderson2019,Chen2024RobustSLS,Leeman2025Robust,Bartos2026Stochastic}, we develop a fast feedback synthesis formulation that explicitly preserves the cross-time correlation structure. The resulting synthesis problem is decomposed into multiple column-wise affine recursions sharing a single common Riccati recursion, thereby enabling highly real-time covariance shaping under persistent parametric uncertainty. Second, a task-oriented terminal covariance metric and weighting construction method is developed. By mapping terminal state errors to the equivalent corrective impulse required for subsequent orbital correction, an analytical task-oriented terminal weight is derived. This enables the feedback synthesis to purposefully shape terminal uncertainty according to the post-injection correction demand, improving both iterative convergence and terminal injection accuracy. Finally, a real-time optimization framework is established and validated for joint stochastic trajectory–guidance replanning. Nominal trajectory planning and feedback synthesis are alternated through uncertainty propagation and chance-constraint tightening. The proposed SLS formulation for persistent parametric uncertainty and task-oriented terminal weighting are incorporated into the feedback synthesis. Innovation-based response realization and energy-based control interpolation are used to enable efficient online implementation and reduce sensitivity to burn-time variation, respectively. Nonlinear Monte Carlo simulations under heterogeneous uncertainties demonstrate subsecond solution times, low constraint violation rates, and high orbital injection accuracy.

The remainder of the paper is organized as follows. Sec.~II formulates the nominal ascent trajectory optimization problem, defines the uncertainty model, and introduces the local chance-constraint approximation. Sec.~III develops the fast SLS formulation for persistent parametric uncertainty. Sec.~IV establishes the task-oriented stochastic replanning framework, including the post-injection correction metric, task-oriented terminal weight, alternating stochastic replanning procedure, and online feedback realization. Sec.~V presents the numerical experiments and nonlinear Monte Carlo results, followed by the conclusions in Sec.~VI.

\section{Stochastic Ascent Trajectory Replanning Problem}\label{sec:problem}

This section formulates the post-fault ascent trajectory
replanning problem for launch vehicles. Sec.~\ref{sec:ascent_dynamics} describes the nominal ascent dynamics. Sec.~\ref{sec:nominal_replanning} introduces the chance constraints and the corresponding nominal replanning problem. Sec.~\ref{sec:uncertainty_model} specifies the uncertainty model used in the optimization.

\subsection{Nominal Ascent Dynamics}\label{sec:ascent_dynamics}
The Earth-centered inertial (ECI) frame
$\mathrm{X}_{\mathrm{I}}$-$\mathrm{Y}_{\mathrm{I}}$-$\mathrm{Z}_{\mathrm{I}}$
and the launch inertial frame
$\mathrm{X}_{\mathrm{L}}$-$\mathrm{Y}_{\mathrm{L}}$-$\mathrm{Z}_{\mathrm{L}}$
used in this paper are shown in Fig.~\ref{fig:coordinates}.
The $\mathrm{Y}_{\mathrm{L}}$ axis is directed radially outward from
the Earth at the launch site. The $\mathrm{X}_{\mathrm{L}}$ axis lies
in the local horizontal plane and points along the prescribed launch
direction, while the $\mathrm{Z}_{\mathrm{L}}$ axis follows the
right-handed convention.

After a propulsion system fault has been diagnosed, the deterministic
replanner uses the identified mean post-fault propulsion capability.
The physical state and thrust-direction command are
$x=[{r}^{\mathsf T}\ {v}^{\mathsf T}\ m]^{\mathsf T}$ and
$u=[\phi\ \psi]^{\mathsf T}$, where ${r}$ and ${v}$ are the
position and velocity expressed in the ECI frame, $m$ is vehicle
mass, and $\phi$ and $\psi$ are the commanded pitch and yaw angles,
respectively, defined in the launch inertial frame. The mean post-fault dynamics in phase $i$ are
\begin{subequations}\label{eq:nominal_dynamics}
\begin{align}
 \dot{{r}}
 &={v},\\
 \dot{{v}}
 &=-\frac{\mu{r}}{\|{r}\|^3}
 -\frac{\rho(h)\bar C_{\mathrm{D},i}S}{2m}
   \|{v}_{\mathrm{rel}}\|{v}_{\mathrm{rel}}
 +\frac{\bar T_i(p_{\mathrm a}(h))}{m}{e}_{T}(u),\\
 \dot m
 &=\dot{\bar m}_i,
\end{align}
\end{subequations}
where $h=\|{r}\|-R_{\mathrm{E}}$ is the altitude,
${v}_{\mathrm{rel}}={v}-{\omega}_{\mathrm{E}}\times{r}$
is the atmosphere-relative velocity, and $\mu$, $R_{\mathrm{E}}$,
and ${\omega}_{\mathrm{E}}$ denote the Earth's gravitational
parameter, radius, and angular velocity vector, respectively.
The thrust direction is
${e}_T(u)=R_{\mathrm{IL}}
[\cos\phi\cos\psi,\ \sin\phi\cos\psi,\ -\sin\psi]^{\mathsf T}$,
where $R_{\mathrm{IL}}$ is the constant rotation matrix mapping vector
components from the launch inertial frame to the ECI frame.
Both atmospheric density $\rho(h)$ and ambient pressure $p_{\mathrm a}(h)$
are obtained from the 1976 U.S. Standard Atmosphere (US76) model.
Here, $S$ is the reference area and $\bar C_{\mathrm{D},i}$ is the
mean drag coefficient in phase $i$. The identified mean vacuum thrust
$\bar T_{i,\mathrm{vac}}$ and vacuum specific impulse
$\bar I_{\mathrm{sp},i}$ are treated as constant
propulsion parameters. The thrust and mass-flow rate are modeled as
$\bar T_i(p_{\mathrm a}(h))=\bar T_{i,\mathrm{vac}}-p_{\mathrm a}(h)A_{\mathrm{e},i}$ and
$\dot{\bar m}_i=-\bar T_{i,\mathrm{vac}}/
(\bar I_{\mathrm{sp},i}g_0)$, respectively, where
$A_{\mathrm{e},i}$ is the nozzle exit area, and $g_0$ is the
standard gravitational acceleration.

The boundary conditions and path constraints are summarized as
\begin{subequations}
\begin{gather}
x(t_0)=x_0,\qquad
\psi_{\mathrm{f}}(x(t_{\mathrm{f}}))=0,\qquad
\label{eq:boundary_constraints}\\
g_\mathrm{b}(x_{\mathrm s},t_{\mathrm s})\le0,\quad
q_{\mathrm h}(t)\le q_{\mathrm h,\max},\quad
q_{\mathrm d}(t)\le q_{\mathrm d,\max},\quad
|\dot u(t)|\le\dot u_{\max},\quad
m(t)\ge m_{\min},\quad
t\in[t_0,t_{\mathrm{f}}],
\label{eq:path_constraints}
\end{gather}
\end{subequations}
where $t_0$ and $t_{\mathrm{f}}$ are the initial and terminal times,
respectively, and $\psi_{\mathrm{f}}$ specifies the terminal orbital
conditions following Ref.~\cite{Miao2022}. Here,
$q_{\mathrm d}=\frac{1}{2}\rho(h)\|{v}_{\mathrm{rel}}\|^2$ is the dynamic
pressure, and heat flux $q_{\mathrm h}$ is estimated using the Sutton--Graves and
free molecular flow correlations in the continuum and rarefied
regimes, respectively, connected by the bridging relation in
Ref.~\cite{Matting1971Approximate}. The limits $q_{\mathrm d,\max}$ and
$q_{\mathrm h,\max}$ bound the dynamic pressure and heat flux, while
$\dot u_{\max}$ and $m_{\min}$ denote the maximum thrust direction
rate and minimum vehicle mass, respectively. The constraint $g_\mathrm{b}$
enforces impact-location limits through return propagation from
the separation state $x_{\mathrm s}$ at time $t_{\mathrm s}$, following
Refs.~\cite{Benedikter2021Autonomous,Benedikter2022Stochastic},
with a downrange-aligned rectangular impact region and the same
atmospheric density model and disturbance configuration as the
ascent phase.

\begin{figure}[t]
\centering
\includegraphics[width=0.45\columnwidth]{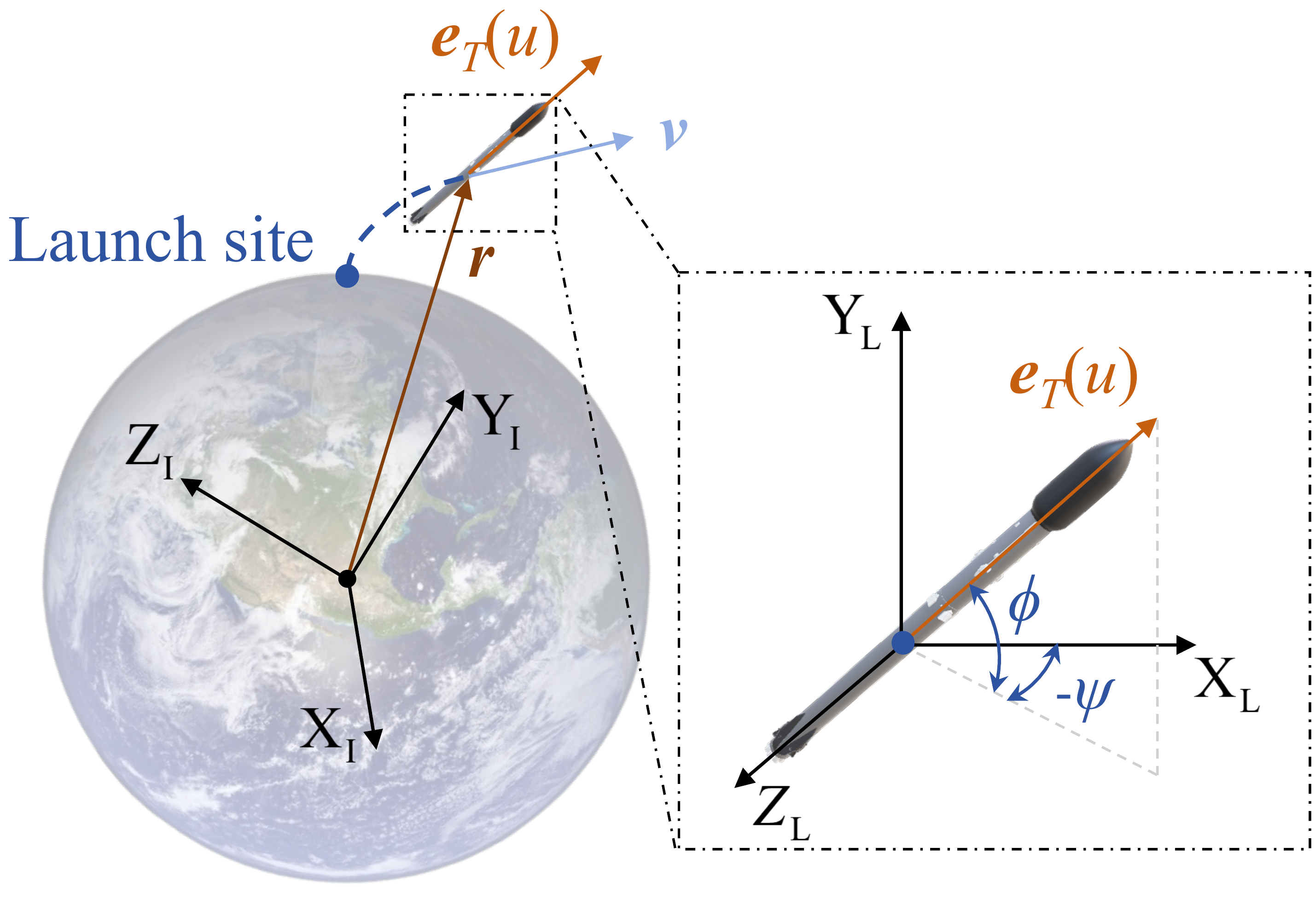}
\caption{Coordinate frames and angle definitions used in the ascent trajectory replanning model.}
\label{fig:coordinates}
\end{figure}

\subsection{Chance-Constrained Nominal Replanning}\label{sec:nominal_replanning}

For each scalar constraint to be enforced probabilistically, let \(g_r(\zeta_r)\le0\), where \(\zeta_r\) collects the physical quantities on which the \(r\)th constraint depends. The corresponding marginal chance constraint is \(\Pr\{g_r(\zeta_r)\le0\}\ge1-\pi_r\), where \(\pi_r\) is the allowable violation probability. Linearizing about the nominal trajectory gives \(G_r=\nabla_{\zeta_r}g_r(\bar\zeta_r)^{\mathsf T}\), where \(\bar{\zeta}_r\) denotes the nominal value of \(\zeta_r\). Under the local Gaussian approximation, the constraint is tightened as
\begin{equation}
g_r(\bar\zeta_r)+\beta_r\le0,\qquad
\beta_r=\kappa_r\sqrt{G_rP_{\zeta_r}G_r^{\mathsf T}},
\label{eq:chance_backoff}
\end{equation}
where \(\kappa_r=\mathcal F_{\mathrm N}^{-1}(1-\pi_r)\), \(\mathcal F_{\mathrm N}\) is the standard normal cumulative distribution function, and \(P_{\zeta_r}\) is the covariance of the corresponding physical variables supplied by the feedback synthesis in Sec.~\ref{sec:sls}.

For a given set of chance backoffs, the nominal replanning problem is
\begin{equation}
\begin{aligned}
\mathcal P_1:\qquad
\min_{x(\cdot),u(\cdot)}
&J_{\mathrm{traj}}[x(\cdot),u(\cdot)]\\
\text{s.t.}\quad
&\text{Eq.~\eqref{eq:nominal_dynamics}, \eqref{eq:boundary_constraints} and \eqref{eq:chance_backoff}.}
\end{aligned}
\label{eq:nominal_replan}
\end{equation}
where the trajectory objective $J_{\mathrm{traj}}$ follows the design in Ref.~\cite{Miao2022}. The resulting nonlinear optimal control problem is solved using sequential convex programming (SCP) \cite{Malyuta2022Tutorial,Pavanello2025CollisionAvoidance,Malyuta2021Advances}.

\subsection{Uncertainty Model}\label{sec:uncertainty_model}

Let $N$ denote the number of discrete propagation steps. The state and control values are defined at nodes \(k=0,\ldots,N\), while the propagation intervals are indexed by \(k=0,\ldots,N-1\). For a powered interval $k$ in phase $i$, the uncertain parameters are modeled
as
\begin{equation}
\begin{aligned}
T_k &= \bar T_k(1+b_{\mathrm T}^{(i)}+\eta_{\mathrm T,k}), \qquad
\dot m_k = \dot{\bar m}_i(1+b_{\dot m}^{(i)}+\eta_{\dot m,k}),\\
C_{\mathrm D,k} &= \bar C_{\mathrm D,i}(1+b_{C_{\mathrm D}}^{(i)}+\eta_{C_{\mathrm D},k}), \qquad
\rho_k = \bar\rho_k(1+b_\rho^{(i)}+\eta_{\rho,k}),
\end{aligned}
\label{eq:uncertain_parameters}
\end{equation}
where $b^{(i)}$ and $\eta_k$ denote the phase-wise persistent relative
biases and interval-wise relative fluctuations, respectively. Collect the relative errors into
\begin{equation}
\omega_k=\operatorname{col}\big(
b_{\mathrm T}^{(i)}+\eta_{\mathrm T,k},\
b_{\dot m}^{(i)}+\eta_{\dot m,k},\
b_{C_{\mathrm D}}^{(i)}+\eta_{C_{\mathrm D},k},\
b_\rho^{(i)}+\eta_{\rho,k}
\big)=L_{b}^{(i)}\xi_b+L_{\eta}^{(i)}\xi_k.
\label{eq:uncertainty_vector}
\end{equation}
where $\xi_b\sim\mathcal N(0,I_{n_b})$ and
$\xi_k\sim\mathcal N(0,I_4)$ are independent. Here, $\xi_{b}$ contains the persistent standardized coordinates shared across all applicable intervals, whereas $\xi_k$ is independently
redrawn at each interval. The phase-dependent matrix
$L_{b}^{(i)}\in\mathbb R^{4\times n_{b}}$ maps the global
persistent coordinates to the physical uncertainty vector active in
phase $i$, while $L_\eta^{(i)}\in\mathbb R^{4\times4}$ maps the
interval-wise independent uncertainty.

For either source class $\circ\in\{b,\eta\}$, let the covariance
matrix satisfy
\begin{equation}
\begin{aligned}
\Sigma_\circ^{(i)}
=L_\circ^{(i)}(L_\circ^{(i)})^{\mathsf T}
=\operatorname{blkdiag}\left(
\begin{bmatrix}
(\sigma_{\circ,T}^{(i)})^2 &
\varrho_{\circ,T\dot m}^{(i)}
\sigma_{\circ,T}^{(i)}\sigma_{\circ,\dot m}^{(i)}\\
\varrho_{\circ,T\dot m}^{(i)}
\sigma_{\circ,T}^{(i)}\sigma_{\circ,\dot m}^{(i)} &
(\sigma_{\circ,\dot m}^{(i)})^2
\end{bmatrix},
(\sigma_{\circ,C_{\mathrm D}}^{(i)})^2,
(\sigma_{\circ,\rho}^{(i)})^2
\right),
\end{aligned}
\label{eq:source_covariance}
\end{equation}
where $\sigma_{\circ,\cdot}^{(i)}$ denotes the standard deviation of
the corresponding relative error, and
$\varrho_{\circ,T\dot m}^{(i)}$ is the correlation coefficient between
the thrust and mass flow rate relative errors.

The state error at replanning activation is modeled as
$\delta x_0\sim\mathcal N(0,P_0)$
and is independent of the propulsion and aerodynamic uncertainty sources. The initial uncertainty can also account for a finite replanning delay between fault diagnosis and activation of the replanned trajectory. 

To improve numerical conditioning, all equations and optimization formulations in this paper are globally nondimensionalized using characteristic length, time, and mass scales, consistent with the scaling practice in Ref.~\cite{Pilipovsky2025FDOC}. 

\section{Fast System Level Synthesis for Persistent Parametric Uncertainty}\label{sec:sls}

For a fixed nominal trajectory from $\mathcal P_1$, this section develops a fast feedback synthesis method that accounts for persistent parametric uncertainty and provides the covariance required by Eq.~\eqref{eq:chance_backoff}. Sec.~\ref{sec:foh_feedback} develops the FOH perturbation model with energy-based nodes and formulates the quadratic feedback objective. Sec.~\ref{sec:slp} introduces the finite-horizon system level parameterization, and Sec.~\ref{sec:persistent} develops the persistent-correlation response optimization and its Riccati-based solution. Sec.~\ref{sec:parallel} analyzes the computational complexity and parallel structure.

\subsection{FOH Model with Energy-Based Nodes and Feedback Objective}\label{sec:foh_feedback}

Define the specific mechanical energy as $\varepsilon(x)=\|{v}\|^2/2-\mu/\|{r}\|$. Given the nominal trajectory $\{(\bar x_k,\bar u_k)\}_{k=0}^{N}$, we use the energy levels $\varepsilon_k=\varepsilon(\bar x_k)$ at the nominal nodes to define the boundaries of the feedback intervals. Under uncertainties, the perturbed trajectory generally reaches these energy levels at times different from the corresponding nominal times, allowing the interval durations to adapt to variations in energy accumulation.

Let $t_k$ denote the nominal time at which the trajectory reaches the $k$th energy node, and define the nominal interval duration as $\tau_k=t_{k+1}-t_k$. Under a first-order hold (FOH) parameterization of the control input over interval $k$, let $x_{k+1}=F_k(x_k,u_k,u_{k+1},\tau_k,\omega_k)$ denote the nonlinear state-transition function. Let $\delta t_k$ denote the deviation of the actual event time from $t_k$, so that the interval-duration deviation is $\delta\tau_k=\delta t_{k+1}-\delta t_k$. Following the standard FOH discretization in Refs.~\cite{Malyuta2021Advances,Malyuta2022Tutorial}, first-order variational linearization about the nominal trajectory yields
\begin{equation}
\delta x_{k+1}
=\delta\hat x_{k+1}^{-}+c_k\delta\tau_k,\qquad
\delta\hat x_{k+1}^{-}
=A_k\delta x_k+B_k^-\delta u_k
+B_k^+\delta u_{k+1}+D_k\omega_k,
\label{eq:foh_physical}
\end{equation}
where $\delta\hat x_{k+1}^{-}$ is the predicted state perturbation over the fixed nominal duration $\tau_k$, prior to event-time correction, and $c_k=\partial F_k/\partial\tau_k$ is the interval-duration sensitivity evaluated along the nominal trajectory.

Define the augmented perturbation state as $\chi_k=[\delta x_k^{\mathsf T}\ \delta u_k^{\mathsf T}\ \delta t_k]^{\mathsf T}\in\mathbb R^{10}$ and the FOH increment input as $\nu_k=\delta u_{k+1}-\delta u_k\in\mathbb R^2$. The initial augmented covariance is factored as $P_{\chi,0}=L_{-1}L_{-1}^{\mathsf T}$, with $\chi_0=L_{-1}\xi_{-1}$ and $\xi_{-1}\sim\mathcal N(0,I_7)$, where $L_{-1}\in\mathbb R^{10\times7}$ and the initial control and event-time deviations are zero.

Since the next node must lie on the prescribed energy surface, its state perturbation satisfies $n_{\varepsilon,k+1}^{\mathsf T}\delta x_{k+1}=0$, where $n_{\varepsilon,k+1}=\nabla_x\varepsilon(\bar x_{k+1})$. Substituting Eq.~\eqref{eq:foh_physical} into this condition gives the interval-duration correction $\delta\tau_k=-n_{\varepsilon,k+1}^{\mathsf T}\delta\hat x_{k+1}^{-} / n_{\varepsilon,k+1}^{\mathsf T}c_k$. Because the specific mechanical energy increases strictly along the nominal ascent trajectory, the energy surfaces are crossed transversely, ensuring that $n_{\varepsilon,k+1}^{\mathsf T}c_k\ne0$. This first-order event-time correction follows the standard treatment of event-driven systems \cite{Kong2024Saltation}.

Before event-time correction, the augmented perturbation $\hat\chi_{k+1}^{-}=[(\delta\hat x_{k+1}^{-})^{\mathsf T}\ \delta u_{k+1}^{\mathsf T}\ \delta t_k]^{\mathsf T}$ satisfies $\hat\chi_{k+1}^{-}=\hat A_k\chi_k+\hat B_k\nu_k+\hat D_k\omega_k$, where
\begin{equation}
\hat A_k=
\begin{bmatrix}
A_k & B_k^-+B_k^+ & 0\\
0 & I_2 & 0\\
0 & 0 & 1
\end{bmatrix},\qquad
\hat B_k=
\begin{bmatrix}
B_k^+\\ I_2\\ 0
\end{bmatrix},\qquad
\hat D_k=
\begin{bmatrix}
D_k\\ 0\\ 0
\end{bmatrix}.
\label{eq:pre_event_matrices}
\end{equation}
Define $\hat c_k=\operatorname{col}(c_k,0,0,1)$ and $\hat n_{\varepsilon,k+1}=\operatorname{col}(n_{\varepsilon,k+1},0,0,0)$. The event-time correction can then be incorporated through the projection matrix $\mathcal T_k^\varepsilon$ and the resulting system matrices
\begin{equation}
\mathcal T_k^\varepsilon=I_{10}-\frac{\hat c_k\hat n_{\varepsilon,k+1}^{\mathsf T}}{\hat n_{\varepsilon,k+1}^{\mathsf T}\hat c_k},
\qquad
\chi_{k+1}=\mathcal T_k^\varepsilon\hat\chi_{k+1}^{-},
\qquad
(\widetilde A_k,\widetilde B_k,\widetilde D_k)
=\mathcal T_k^\varepsilon(\hat A_k,\hat B_k,\hat D_k).
\label{eq:event_projection}
\end{equation}
This projection simultaneously projects the physical-state perturbation onto the tangent space of the prescribed energy surface and updates the event-time deviation via $\delta t_{k+1}=\delta t_k+\delta\tau_k$. The resulting dynamics with energy-based nodes are
\begin{equation}
\chi_{k+1}
=\widetilde A_k\chi_k+\widetilde B_k\nu_k+\widetilde D_k\omega_k
=\widetilde A_k\chi_k+\widetilde B_k\nu_k
+\widetilde D_kL_{b,k}\xi_{b}
+\widetilde D_kL_{\eta,k}\xi_k.
\label{eq:augmented}
\end{equation}
where the phase-dependent uncertainty-loading matrices $L_{b}^{(i)}$ and $L_\eta^{(i)}$ are denoted by $L_{b,k}$ and $L_{\eta,k}$ according to the active phase since each global interval $k$ belongs to a unique phase. The persistent standardized coordinates $\xi_{b}$ are shared across all applicable intervals, whereas the independent coordinates $\xi_k$ are independently redrawn for each interval.

The local feedback objective is
\begin{equation}
J_{\mathrm c}=
\mathbb E\!\left[
\delta x_{\mathrm f}^{\mathsf T}Q_{\mathrm f}\delta x_{\mathrm f}
+\int_{t_0}^{t_{\mathrm f}}
\left(
\delta x^{\mathsf T}Q_{\mathrm c}\delta x+
\delta u^{\mathsf T}R_{\mathrm c}\delta u
\right)\,dt
\right],
\label{eq:continuous_cost}
\end{equation}
where $Q_{\mathrm c}\succeq0$ and $R_{\mathrm c}\succ0$ are the continuous-time state and control weights, respectively, and $Q_{\mathrm f}\succeq0$ is the terminal penalty matrix. The running cost penalizes event-aligned state and control deviations, without a separate penalty on the event-time deviation. The stage matrices are constructed using left-endpoint quadrature for the state term and exact integration of the FOH control-deviation term over each nominal interval. Since the running-cost density is second order and the interval-duration deviation is first order, their product is third order and is omitted from the local quadratic objective. The augmented terminal weight is $\widetilde Q_{\mathrm f}=\operatorname{blkdiag}(Q_{\mathrm f},0_{3\times3})\succeq0$, which penalizes the physical terminal-state deviation without directly penalizing the control or event-time deviation. The discrete feedback objective is therefore\
\begin{equation}
J_{\mathrm{phys}}
=\mathbb E\!\left[
\chi_N^{\mathsf T}\widetilde Q_{\mathrm f}\chi_N
+\sum_{k=0}^{N-1}
[\chi_k^{\mathsf T}\ \nu_k^{\mathsf T}]
W_k
[\chi_k^{\mathsf T}\ \nu_k^{\mathsf T}]^{\mathsf T}
\right],\qquad
W_k=
\begin{bmatrix}
Q_k&S_k\\
S_k^{\mathsf T}&R_k
\end{bmatrix}.
\label{eq:discrete_cost}
\end{equation}

\subsection{System Level Parameterization}\label{sec:slp}

Introduce the abstract disturbance $d_k=\widetilde D_k\omega_k$ and $d_{-1}=\chi_0$, so that $\chi_{k+1}=\widetilde A_k\chi_k+\widetilde B_k\nu_k+d_k$. Let $n_\chi=\dim(\chi_k)$ and $n_\nu=\dim(\nu_k)$. For causal responses, let $\Phi^x_{k,j}\in\mathbb R^{n_\chi\times n_\chi}$ and $\Phi^u_{k,j}\in\mathbb R^{n_\nu\times n_\chi}$ denote the state- and input-response matrices to disturbance block $d_j$, respectively. Then
\begin{equation}\label{eq:response_sum}
\chi_k=\Phix_{k,-1}d_{-1}+\sum_{j=0}^{k-1}\Phix_{k,j}d_j,\qquad
\nu_k=\Phiu_{k,-1}d_{-1}+\sum_{j=0}^{k-1}\Phiu_{k,j}d_j.
\end{equation}

Define the stacked vectors $\bm\chi=\col(\chi_0,\ldots,\chi_N)$, $\bm\nu=\col(\nu_0,\ldots,\nu_{N-1})$, $\bm d=\col(d_{-1},d_0,\ldots,d_{N-1})$, and the full-horizon response variable
\begin{equation}
\bm z=\col(\bm\chi,\bm\nu)=\bm\Phi\bm d, \qquad \bm\Phi = [(\bm\Phi^x)^{\mathsf T}, (\bm\Phi^u)^{\mathsf T}]^{\mathsf T}.
\label{eq:full_response}
\end{equation}

The corresponding block-lower-triangular response matrices are
\begin{equation}
\bm\Phi^x=
\begin{bmatrix}
I&0&\cdots&0\\
\Phix_{1,-1}&I&\ddots&\vdots\\
\vdots&\ddots&\ddots&0\\
\Phix_{N,-1}&\cdots&\Phix_{N,N-2}&I
\end{bmatrix},\qquad
\bm\Phi^u=
\begin{bmatrix}
\Phiu_{0,-1}&0&\cdots&0\\
\Phiu_{1,-1}&\Phiu_{1,0}&\ddots&\vdots\\
\vdots&\vdots&\ddots&0\\
\Phiu_{N-1,-1}&\cdots&\Phiu_{N-1,N-2}&0
\end{bmatrix}.
\label{eq:response_matrices}
\end{equation}

Let $Z$ denote the block down-shift operator, with $\bm A=\operatorname{blkdiag}(\widetilde A_0,\ldots,\widetilde A_{N-1},0)$ and $\bm B=\operatorname{blkdiag}(\widetilde B_0,\ldots,\widetilde B_{N-1},0)$, including compatible terminal padding, and let $\mathcal Z_{AB}=\left[I-Z\bm A\quad-Z\bm B\right]$. By the finite-horizon system level parameterization of Ref.~\cite[Theorem~2.1]{Wang2019SLS}, a pair of causal response matrices is achievable if and only if $\mathcal Z_{AB}\bm\Phi=I$.
For any feasible response pair, the causal feedback law $\bm\nu=\bm K\bm\chi$ is recovered from 
\begin{equation}
 \bm K\bm\Phi^x=\bm\Phi^u,
 \qquad
 \bm K=\bm\Phi^u(\bm\Phi^x)^{-1}.
 \label{eq:feedback_response_relation}
\end{equation}
The inverse exists because $\bm\Phi^x$ is block lower triangular with identity diagonal blocks.

Substituting the definitions of \(d_j\) and \(\omega_j\) into Eq.~\eqref{eq:response_sum} gives
\begin{equation}
\chi_k=\Psi_{k,-1}^{x}\xi_{-1}
+\Psi_{k,b}^{x}\xi_{b}
+\sum_{j=0}^{k-1}\Psi_{k,j}^{x}\xi_j,
\label{eq:physical_response_state}
\end{equation}
where
\begin{equation}
\Psi_{k,-1}^{x}=\Phix_{k,-1}L_{-1}, \qquad
\Psi_{k,b}^{x}=\sum_{j=0}^{k-1}\Phix_{k,j}\widetilde D_jL_{b,j},\qquad
\Psi_{k,j}^{x}=\Phix_{k,j}\widetilde D_jL_{\eta,j}
\label{eq:physical_response_factors}
\end{equation}
Since \(\xi_{-1}\), \(\xi_b\), and \(\{\xi_j\}_{j=0}^{k-1}\) are mutually independent with identity covariance, the augmented state covariance is
\begin{equation}
P_{\chi,k}=\mathbb E[\chi_k\chi_k^{\mathsf T}]=\Psi_{k,-1}^{x}(\Psi_{k,-1}^{x})^{\mathsf T}
+\Psi_{k,b}^{x}(\Psi_{k,b}^{x})^{\mathsf T}+\sum_{j=0}^{k-1}
\Psi_{k,j}^{x}(\Psi_{k,j}^{x})^{\mathsf T}.
\label{eq:physical_covariance}
\end{equation}

\begin{remark}[Cross-time correlation of persistent parametric uncertainty]\label{rem:cross_time}
For independent disturbances across intervals, the cross terms between responses to different intervals vanish in the covariance, leaving only the sum of individual contributions. In contrast, persistent coordinates $\xi_{b}$ are shared across intervals, so the cross terms between their responses across different intervals remain in the covariance.
\end{remark}

This response-based covariance expression is equivalent to linear covariance propagation of the same closed-loop perturbation model, while explicitly retaining the cross-time correlations induced by the persistent parametric uncertainty. The constraint covariance \(P_{\zeta_r}\) in Eq.~\eqref{eq:chance_backoff} is assembled analogously from the corresponding physical responses. Let \(\xi_r\) collect the mutually independent standardized coordinates relevant to constraint \(r\), and write
\begin{equation}
\delta\zeta_r=\Psi_{\zeta,r}\xi_r,
\qquad
P_{\zeta_r}=\Psi_{\zeta,r}\Psi_{\zeta,r}^{\mathsf T}.
\label{eq:zeta_covariance}
\end{equation}
The state and control components of \(\Psi_{\zeta,r}\) are obtained from the closed-loop responses above, while direct uncertain parameter responses are expressed on the same standardized coordinates, thereby retaining their statistical dependence.

\subsection{Structured Response Optimization under Persistent Parametric Uncertainty}\label{sec:persistent}
Let $\bm W$ be the symmetric full-horizon representation of the stage
and terminal weights, defined by
\begin{equation}
    \bm z^{\mathsf T}\bm W\bm z=
    \left[
    \chi_N^{\mathsf T}\widetilde Q_{\mathrm f}\chi_N
    +\sum_{k=0}^{N-1}
    [\chi_k^{\mathsf T}\ \nu_k^{\mathsf T}]
    W_k
    [\chi_k^{\mathsf T}\ \nu_k^{\mathsf T}]^{\mathsf T}
    \right]
\end{equation}

The abstract disturbance is generated by the physical random
coordinates as $\bm d=\bm D_{\eta}\bm\xi_{\eta}
+\bm G_{b}\xi_{b}$, where
$\bm\xi_{\eta}=\operatorname{col}(\xi_{-1},\xi_0,\ldots,\xi_{N-1})$. Define 

\begin{equation}
    \bm D_{\eta}
    =\operatorname{blkdiag}\left(
    L_{-1},\widetilde D_0L_{\eta,0},\ldots,
    \widetilde D_{N-1}L_{\eta,N-1}\right),\qquad
    \bm G_{b}
    =\operatorname{col}\left(
    0,\widetilde D_0L_{b,0},\ldots,
    \widetilde D_{N-1}L_{b,N-1}\right).
\label{eq:independent_loading}
\end{equation}

Using Eq.~\eqref{eq:full_response}, the physical objective in
Eq.~\eqref{eq:discrete_cost} becomes
\begin{equation}
    J_{\mathrm{phys}}
    =\E[\bm z^{\mathsf T}\bm W\bm z]
    =\operatorname{tr}\bigl(
    \bm\Phi^{\mathsf T}\bm W\bm\Phi
    \bm M_{\mathrm{phys}}\bigr),
    \label{eq:physical_objective}
\end{equation}
where $\bm M_{\mathrm{phys}}=\E[\bm d\bm d^{\mathsf T}]
=\bm D_{\eta}\bm D_{\eta}^{\mathsf T}
+\bm G_{b}\bm G_{b}^{\mathsf T}$.

The original SLS problem is 
\begin{subequations}\label{prob:physical_synthesis}
\begin{align}
    \mathcal P_{2\mathrm{a}}:\qquad
    \min_{\bm\Phi\in\mathcal C}\quad&
    J_{\mathrm{phys}}(\bm\Phi)\\
    \mathrm{s.t.}\quad&
    \mathcal Z_{AB}\bm\Phi=I.
\end{align}
\end{subequations}
where $\mathcal C$ denotes the set of admissible causal responses. This is the feedback synthesis problem of interest, formulated in terms of the closed-loop responses. However, its original formulation is difficult to solve efficiently. We therefore construct several auxiliary problems to facilitate its solution.

By construction, $\bm M_{\mathrm{phys}}\succeq0$ is structurally singular. Because neither the initial perturbation $L_{-1}$ nor the projected disturbance matrices $\widetilde D_k=\mathcal T_k^\varepsilon\hat D_k$ directly excite the augmented control coordinates $\delta u_k$, the corresponding rows of $\bm D_\eta$ and $\bm G_{b}$ are identically zero. To ensure strict positive definiteness and enable a valid Cholesky factorization, we define the regularized covariance matrix
$
    \bm M_\alpha
    =
    \bm M_{\mathrm{phys}}
    +
    \alpha I
    \succ0
$
for $\alpha>0$. Define
\begin{equation}
    J_\alpha(\bm\Phi)
    =
    \tr\!\left(
    \bm\Phi^{\mathsf T}\bm W\bm\Phi
    \bm M_\alpha
    \right)
    =
    J_{\mathrm{phys}}(\bm\Phi)
    +
    \alpha J_{\mathrm{reg}}(\bm\Phi),
    \qquad
    J_{\mathrm{reg}}(\bm\Phi)
    =
    \tr\!\left(
    \bm\Phi^{\mathsf T}\bm W\bm\Phi
    \right).
    \label{eq:regularized_objective}
\end{equation}
The corresponding regularized auxiliary problem is
\begin{subequations}\label{prob:regularized_synthesis}
\begin{align}
    \mathcal P_{2\mathrm{b}}:\qquad
    \min_{\bm\Phi\in\mathcal C}\quad&
    J_\alpha(\bm\Phi)\\
    \mathrm{s.t.}\quad&
    \mathcal Z_{AB}\bm\Phi=I.
\end{align}
\end{subequations}

Since $\bm M_\alpha\succ0$, let
$
    \bm M_\alpha
    =
    \bm C_\alpha
    \bm C_\alpha^{\mathsf T}
$
be its Cholesky factorization and define
\begin{equation}
    \widehat{\bm\Phi}
    =
    \bm\Phi\bm C_\alpha.
    \label{eq:canonical_transform}
\end{equation}
The transformed auxiliary problem is
\begin{subequations}\label{prob:transformed_synthesis}
\begin{align}
    \mathcal P_{2\mathrm{c}}:\qquad
    \min_{\widehat{\bm\Phi}\in\mathcal C}\quad&
    \tr\!\left(
    \widehat{\bm\Phi}^{\mathsf T}
    \bm W
    \widehat{\bm\Phi}
    \right)\\
    \mathrm{s.t.}\quad&
    \mathcal Z_{AB}\widehat{\bm\Phi}
    =
    \bm C_\alpha.
\end{align}
\end{subequations}

\begin{theorem}[Fast System Level Synthesis for Persistent Parametric Uncertainty]\label{thm:structured}
Let $\bm\Phi^\star$ be an optimal solution of $\mathcal P_{2\mathrm{a}}$ and let
$\widehat{\bm\Phi}_\alpha^\star$ be an optimal solution of $\mathcal P_{2\mathrm{c}}$. Define
$\bm\Phi_\alpha^\star=\widehat{\bm\Phi}_\alpha^\star\bm C_\alpha^{-1}.$
Suppose the stage and terminal weighting matrices satisfy
\begin{equation}
    W_k\succeq0,
    \qquad
    R_k\succ0,
    \quad
    k=0,\ldots,N-1,
    \qquad
    \widetilde Q_{\mathrm{f}}\succeq0.
    \label{eq:theorem_weight_conditions}
\end{equation}
Then $\bm\Phi_\alpha^\star$ is feasible for $\mathcal P_{2\mathrm{a}}$ and satisfies
\begin{equation}
    0
    \le
    J_{\mathrm{phys}}(\bm\Phi_\alpha^\star)
    -
    J_{\mathrm{phys}}(\bm\Phi^\star)
    \le
    \alpha J_{\mathrm{reg}}(\bm\Phi^\star).
    \label{eq:regularization_bound}
\end{equation}
Moreover, $\widehat{\bm\Phi}_\alpha^\star$ can be obtained from one common Riccati backward recursion together with mutually independent affine recursions for the $N+1$ transformed disturbance columns.
\end{theorem}

\begin{proof}
Since $\bm M_\alpha\succ0$, its Cholesky factor $\bm C_\alpha$ is nonsingular. Both
$\bm C_\alpha$ and $\bm C_\alpha^{-1}$ are block lower triangular. Therefore,
right multiplication by $\bm C_\alpha$ defines a bijection on $\mathcal C$. In addition,
$
    \tr\!\left(
    \bm\Phi^{\mathsf T}\bm W\bm\Phi
    \bm M_\alpha
    \right)
    =
    \tr\!\left(
    \widehat{\bm\Phi}^{\mathsf T}
    \bm W
    \widehat{\bm\Phi}
    \right),
    $
and right multiplication of
$\mathcal Z_{AB}\bm\Phi=I$ by $\bm C_\alpha$ gives
$\mathcal Z_{AB}\widehat{\bm\Phi}=\bm C_\alpha$.
Hence $\mathcal P_{2\mathrm{b}}$ and $\mathcal P_{2\mathrm{c}}$ are equivalent, and the inverse image
$\bm\Phi_\alpha^\star
=
\widehat{\bm\Phi}_\alpha^\star\bm C_\alpha^{-1}$
is an optimal solution of $\mathcal P_{2\mathrm{b}}$.

Problems $\mathcal P_{2\mathrm{a}}$ and $\mathcal P_{2\mathrm{b}}$ have the same feasible set because they share the decision set
$\mathcal C$ and the constraint
$\mathcal Z_{AB}\bm\Phi=I$.
Therefore, $\bm\Phi_\alpha^\star$ is also feasible for $\mathcal P_{2\mathrm{a}}$.
Moreover, Eq.~\eqref{eq:theorem_weight_conditions} implies
$\bm W\succeq0$, and hence
$J_{\mathrm{reg}}(\bm\Phi)\ge0$.
Optimality of $\bm\Phi^\star$ for $\mathcal P_{2\mathrm{a}}$ gives
\begin{equation}
    J_{\mathrm{phys}}(\bm\Phi^\star)
    \le
    J_{\mathrm{phys}}(\bm\Phi_\alpha^\star),
    \label{eq:physical_lower_bound}
\end{equation}
whereas optimality of $\bm\Phi_\alpha^\star$ for $\mathcal P_{2\mathrm{b}}$ gives
\begin{equation}
    J_{\mathrm{phys}}(\bm\Phi_\alpha^\star)
    +
    \alpha J_{\mathrm{reg}}(\bm\Phi_\alpha^\star)
    \le
    J_{\mathrm{phys}}(\bm\Phi^\star)
    +
    \alpha J_{\mathrm{reg}}(\bm\Phi^\star).
    \label{eq:regularized_optimality}
\end{equation}
Combining Eqs.~\eqref{eq:physical_lower_bound} and
\eqref{eq:regularized_optimality} yields Eq.~\eqref{eq:regularization_bound}.

It remains to derive the solution structure of $\mathcal P_{2\mathrm{c}}$. Let $\mathcal J=\{-1,0,\ldots,N-1\}$ , $r_j=j+1$ and let
$C_{k,j}\in\mathbb R^{n_\chi\times n_\chi}$ denote the $(k,j)$ block of
$\bm C_\alpha$, and let
$\widehat\Phi^x_{k,j}\in\mathbb R^{n_\chi\times n_\chi}$ and
$\widehat\Phi^u_{k,j}\in\mathbb R^{n_\nu\times n_\chi}$
denote the corresponding transformed state- and input-response blocks, respectively.
Because $\bm C_\alpha$ is block lower triangular,
$C_{k,j}=0$ for $k<r_j$. For each $j\in\mathcal J$, the $j$th block column of the constraint in
$\mathcal P_{2\mathrm{c}}$ satisfies
\begin{equation}
    \widehat\Phi^x_{r_j,j}
    =
    C_{r_j,j},
    \qquad
    \widehat\Phi^x_{k+1,j}
    =
    \widetilde A_k\widehat\Phi^x_{k,j}
    +
    \widetilde B_k\widehat\Phi^u_{k,j}
    +
    C_{k+1,j},
    \quad
    k=r_j,\ldots,N-1.
    \label{eq:column_affine_dynamics}
\end{equation}
If \(r_j=N\), the recursion over \(k=r_j,\ldots,N-1\) is vacuous and column \(j\) contributes only a terminal cost. More generally, the objective of \(\mathcal P_{2\mathrm c}\) can be written as the sum of the quadratic costs associated with its block columns, while the transformed closed-loop response dynamics in Eq.~\eqref{eq:column_affine_dynamics} are likewise separable by block column. Hence, for each \(j\in\mathcal J\), the corresponding transformed disturbance column defines a finite-horizon affine linear-quadratic (LQ) problem. All columns share the same quadratic stage and terminal data, whereas the affine forcing terms \(C_{k,j}\) are column dependent.

The standard finite-horizon LQ dynamic programming argument \cite{athans1966optimal,kwakernaak1972linear} extends directly to these response blocks because both the trace cost and the column dynamics are additive over the columns of the response matrices. The cost-to-go is therefore introduced to separate the common quadratic contribution from the column-dependent affine contribution. For each \(j\in\mathcal J\), the cost-to-go function can be written as
\begin{equation}
 \mathcal V_{k+1,j}\!\left(\widehat\Phi^x_{k+1,j}\right)
 =
 \tr\!\left[
 (\widehat\Phi^x_{k+1,j})^{\mathsf T}
 \mathcal P_{k+1}
 \widehat\Phi^x_{k+1,j}
 \right]
 +2\tr\!\left[
 \Pi_{k+1,j}^{\mathsf T}
 \widehat\Phi^x_{k+1,j}
 \right]
 +s_{k+1,j},
 \label{eq:value_function}
\end{equation}
where
$\mathcal P_{k+1}\in\mathbb R^{n_\chi\times n_\chi}$,
$\Pi_{k+1,j}\in\mathbb R^{n_\chi\times n_\chi}$,
and $s_{k+1,j}\in\mathbb R$ is independent of
$\widehat\Phi^x_{k+1,j}$.
The Bellman recursion for each column is
\begin{equation}
\mathcal V_{k,j}(\widehat\Phi^x_{k,j})
=\min_{\widehat\Phi^u_{k,j}}\left\{
\operatorname{tr}\!\left(
\begin{bmatrix}
\widehat\Phi^x_{k,j}\\
\widehat\Phi^u_{k,j}
\end{bmatrix}^{\mathsf T}
W_k
\begin{bmatrix}
\widehat\Phi^x_{k,j}\\
\widehat\Phi^u_{k,j}
\end{bmatrix}
\right)
+\mathcal V_{k+1,j}\!\left(
\widetilde A_k\widehat\Phi^x_{k,j}
+\widetilde B_k\widehat\Phi^u_{k,j}
+C_{k+1,j}
\right)
\right\}.
\label{eq:bellman_recursion}
\end{equation}

Substituting the quadratic form in Eq.~\eqref{eq:value_function}
into the Bellman recursion in Eq.~\eqref{eq:bellman_recursion}
and carrying out the minimization with respect to
$\widehat\Phi^u_{k,j}$ separates the recursion into a
column-independent quadratic part and column-dependent affine
terms. The former yields the common Riccati recursion
\begin{subequations}\label{eq:riccati_defs}
\begin{align}
    \mathcal P_N&=\widetilde Q_{\mathrm f},\qquad
    \mathcal H_k=R_k+\widetilde B_k^{\mathsf T}
    \mathcal P_{k+1}\widetilde B_k,\qquad
    \mathcal G_k=S_k^{\mathsf T}+\widetilde B_k^{\mathsf T}
    \mathcal P_{k+1}\widetilde A_k,\\
    \mathcal L_k&=-\mathcal H_k^{-1}\mathcal G_k,\qquad
    \mathcal P_k=Q_k+\widetilde A_k^{\mathsf T}
    \mathcal P_{k+1}\widetilde A_k
    -\mathcal G_k^{\mathsf T}\mathcal H_k^{-1}\mathcal G_k,
    \quad k=N-1,\ldots,0.
\end{align}
\end{subequations}
Here
$\mathcal H_k\in\mathbb R^{n_\nu\times n_\nu}$ and
$\mathcal G_k,\mathcal L_k\in\mathbb R^{n_\nu\times n_\chi}$.
The conditions in Eq.~\eqref{eq:theorem_weight_conditions} imply
$\mathcal P_k\succeq0$ and $\mathcal H_k\succ0$ throughout the backward recursion.

Since the terminal cost contains no term linear in \(\widehat\Phi^x_{N,j}\), matching the cost-to-go function to the terminal cost gives \(\Pi_{N,j}=0\). With the common Riccati quantities determined, the column-dependent affine terms are obtained by matching the linear terms in the Bellman recursion. For each \(j\in\mathcal J\) with \(r_j\le N-1\), define \(\Gamma_{k,j}\), \(\Pi_{k,j}\in\mathbb R^{n_\chi\times n_\chi}\) and \(\Lambda_{k,j}\in\mathbb R^{n_\nu\times n_\chi}\) by
\begin{equation}\label{eq:affine_terms}
    \Gamma_{k,j}
    =
    \mathcal P_{k+1}C_{k+1,j}
    +
    \Pi_{k+1,j},\quad
    \Lambda_{k,j}
    =
    -\mathcal H_k^{-1}
    \widetilde B_k^{\mathsf T}
    \Gamma_{k,j},\quad
    \Pi_{k,j}
    =
    \widetilde A_k^{\mathsf T}
    \Gamma_{k,j}
    +
    \mathcal G_k^{\mathsf T}
    \Lambda_{k,j},
    \quad
    k=N-1,\ldots,r_j,
\end{equation}
The minimizing transformed control-response block is
\begin{equation}
    \widehat\Phi^u_{k,j}
    =
    \mathcal L_k
    \widehat\Phi^x_{k,j}
    +
    \Lambda_{k,j},
    \qquad
    k=r_j,\ldots,N-1,
    \label{eq:response_control}
\end{equation}
and
$\widehat\Phi^x_{k,j}$
is propagated forward by Eq.~\eqref{eq:column_affine_dynamics} over the same index range. 

Equation~\eqref{eq:riccati_defs} is independent of $j$ and therefore requires one backward Riccati sweep. Once
$\mathcal P_k$ and $\mathcal L_k$ are available,
Eqs.~\eqref{eq:affine_terms}--\eqref{eq:response_control}
are independent across the transformed disturbance columns.
\end{proof}

Thus, by Theorem~\ref{thm:structured}, $\mathcal P_{2\mathrm{c}}$
can be solved efficiently using one common backward Riccati recursion
and $N+1$ mutually independent affine recursions. Moreover, by choosing
$\alpha$ sufficiently small, the recovered response provides a feasible
solution to the original feedback synthesis problem with an arbitrarily
small suboptimality gap.

The decomposition above also reduces to the standard Riccati state-feedback law, which depends only on the current augmented state, when the persistent uncertainty component is removed.

\begin{corollary}[Reduction under Independent Disturbances]\label{cor:independent}
If $\bm G_{b}=0$, the optimal regularized response reduces to the current augmented-state feedback law
\begin{equation}
    \nu_k
    =
    \mathcal L_k\chi_k,
    \qquad
    k=0,\ldots,N-1.
    \label{eq:current_aug_feedback}
\end{equation}
\end{corollary}

\begin{proof}
When $\bm G_{b}=0$,
$\bm M_\alpha
=
\bm D_{\eta}\bm D_{\eta}^{\mathsf T}
+
\alpha I$
is block diagonal, and so is
$\bm C_\alpha$.
Hence
$C_{k,j}=0$ for $k>r_j$, which gives
$\Pi_{k,j}=\Lambda_{k,j}=0$
over every active column range.
Therefore,
$\bm\Phi^u=\bm L\bm\Phi^x$, where
$\bm L
=
\left[\operatorname{blkdiag}
(\mathcal L_0,\ldots,\mathcal L_{N-1})
\quad 0_{Nn_\nu\times n_\chi}\right]$.
Eq.~\eqref{eq:feedback_response_relation} then gives
$\bm K=\bm L$, yielding
Eq.~\eqref{eq:current_aug_feedback}.
\end{proof}

Persistent parametric uncertainty changes this limit through the cross-time blocks of
$\bm G_{b}\bm G_{b}^{\mathsf T}$.
In this case
$\bm C_\alpha$
is generally not block diagonal, so
$C_{k,j}$ can remain nonzero after the arrival of column $j$.
The resulting
$\Pi_{k,j}$ and
$\Lambda_{k,j}$
are generally nonzero, and the recovered responses need not satisfy
$\bm\Phi^u=\bm L\bm\Phi^x$.
By Eq.~\eqref{eq:feedback_response_relation}, they nevertheless admit a causal feedback operator
$\bm K$ satisfying
$\bm K\bm\Phi^x=\bm\Phi^u$, which may contain strictly historical blocks:
\begin{equation}
    \nu_k
    =
    \sum_{\ell=0}^{k}
    K_{k,\ell}\chi_\ell.
    \label{eq:history_feedback}
\end{equation}
The synthesis uses the cross-time statistical structure of the persistent uncertainty without estimating the realization of \(\xi_{b}\), thereby avoiding dependence on high-accuracy online identification of the residual propulsion uncertainty.

\subsection{Computational Complexity and Parallel Structure}\label{sec:parallel}

In the present ascent application,
$n_\nu\ll n_\chi$, while the dimensions of the persistent and independent
uncertainty coordinates and the number of scalar path constraints per node
are fixed. Accordingly, only the dominant dependence on horizon length $N$ and augmented-state dimension $n_\chi$
is retained below. Table~\ref{tab:complexity} summarizes the principal
feedback computations. 

The common Riccati recursion in
Eq.~\eqref{eq:riccati_defs} requires $O(Nn_\chi^3)$ work. 
The column-dependent affine backward recursions and forward response
propagations in Eqs.~\eqref{eq:column_affine_dynamics},
\eqref{eq:affine_terms}, and \eqref{eq:response_control} require
$O(N^2n_\chi^3)$ serial work. Once the common Riccati matrices and
$\bm C_\alpha$ are available, different transformed columns are independent,
giving an ideal parallel depth of $O(Nn_\chi^3)$. The structured construction of the Cholesky factorization  exploits the low-rank persistent contribution $\bm G_{b}\bm G_{b}^{\mathsf T}$ together with the block-diagonal independent
component. Its serial work is
$O(Nn_\chi^3+N^2n_\chi^2)$. The system responses $\bm\Phi$ are then recovered from
Eq.~\eqref{eq:canonical_transform} by right block-triangular
substitution without explicitly forming $\bm C_\alpha^{-1}$. Recovering all causal
response blocks requires $O(N^3n_\chi^3)$ serial work. Different output
block rows can be processed independently, reducing the ideal parallel
depth to $O(N^2n_\chi^3)$. The recovered responses are used in Eq.~\eqref{eq:physical_response_factors}
to construct the corresponding $\Psi$ matrices.
Equation~\eqref{eq:physical_covariance} then provides the covariance used
to evaluate the chance-constraint backoff in Eq.~\eqref{eq:chance_backoff}. These operations require
$O(N^2n_\chi^2)$ serial work and admit parallel evaluation across nodes
and constraint rows. 

\begin{table}[t]
\caption{Computational complexity of the principal feedback synthesis operations.}
\label{tab:complexity}
\centering
\begin{tabularx}{\linewidth}{@{}
>{\centering\arraybackslash}X
>{\centering\arraybackslash}p{0.18\linewidth}
>{\centering\arraybackslash}p{0.18\linewidth}@{}}
\toprule
Operation
& Single-thread work
& Ideal parallel depth \\
\midrule

Common Riccati sweep,
Eq.~\eqref{eq:riccati_defs}
& $O(Nn_\chi^3)$
& $O(Nn_\chi^3)$\\

Transformed response recursion,
Eqs.~\eqref{eq:column_affine_dynamics},
\eqref{eq:affine_terms} and
\eqref{eq:response_control}
& $O(N^2n_\chi^3)$
& $O(Nn_\chi^3)$ \\

Structured Cholesky factorization $\bm M_\alpha=\bm C_\alpha\bm C_\alpha^{\mathsf T}$
& $O(Nn_\chi^3+N^2n_\chi^2)$
& $O(n_\chi^3+N^2n_\chi^2)$\\

Recovery of $\bm\Phi_\alpha^\star=\widehat{\bm\Phi}_\alpha^\star\bm C_\alpha^{-1}$
& $O(N^3n_\chi^3)$
& $O(N^2n_\chi^3)$\\

Physical covariance and backoff evaluation,
Eqs.~\eqref{eq:chance_backoff}, \eqref{eq:physical_response_factors}
and \eqref{eq:physical_covariance}
& $O(N^2n_\chi^2)$
& $O(Nn_\chi^2)$ \\

\bottomrule
\end{tabularx}
\end{table}

\section{Task-Oriented Stochastic Replanning}\label{sec:task_replanning}

This section completes the proposed joint stochastic trajectory--guidance replanning framework. Sec.~\ref{sec:metric} first introduces a post-injection correction metric for quantifying terminal orbital errors. Sec.~\ref{sec:terminal_weight} then derives a task-oriented terminal weight from this metric, which, together with the system level synthesis in Sec.~\ref{sec:sls}, completes the feedback synthesis. Sec.~\ref{sec:alternating} alternates this feedback synthesis with the nominal replanning problem in Sec.~\ref{sec:problem}, forming the complete stochastic trajectory--guidance replanning procedure. Finally, Sec.~\ref{sec:online_feedback} describes the online realization of the optimized feedback responses using disturbance innovations and energy-based control interpolation.

\subsection{Post-Injection Correction Metric}\label{sec:metric}

\begin{figure}[htbp]
\centering
\includegraphics[width=0.35\columnwidth]{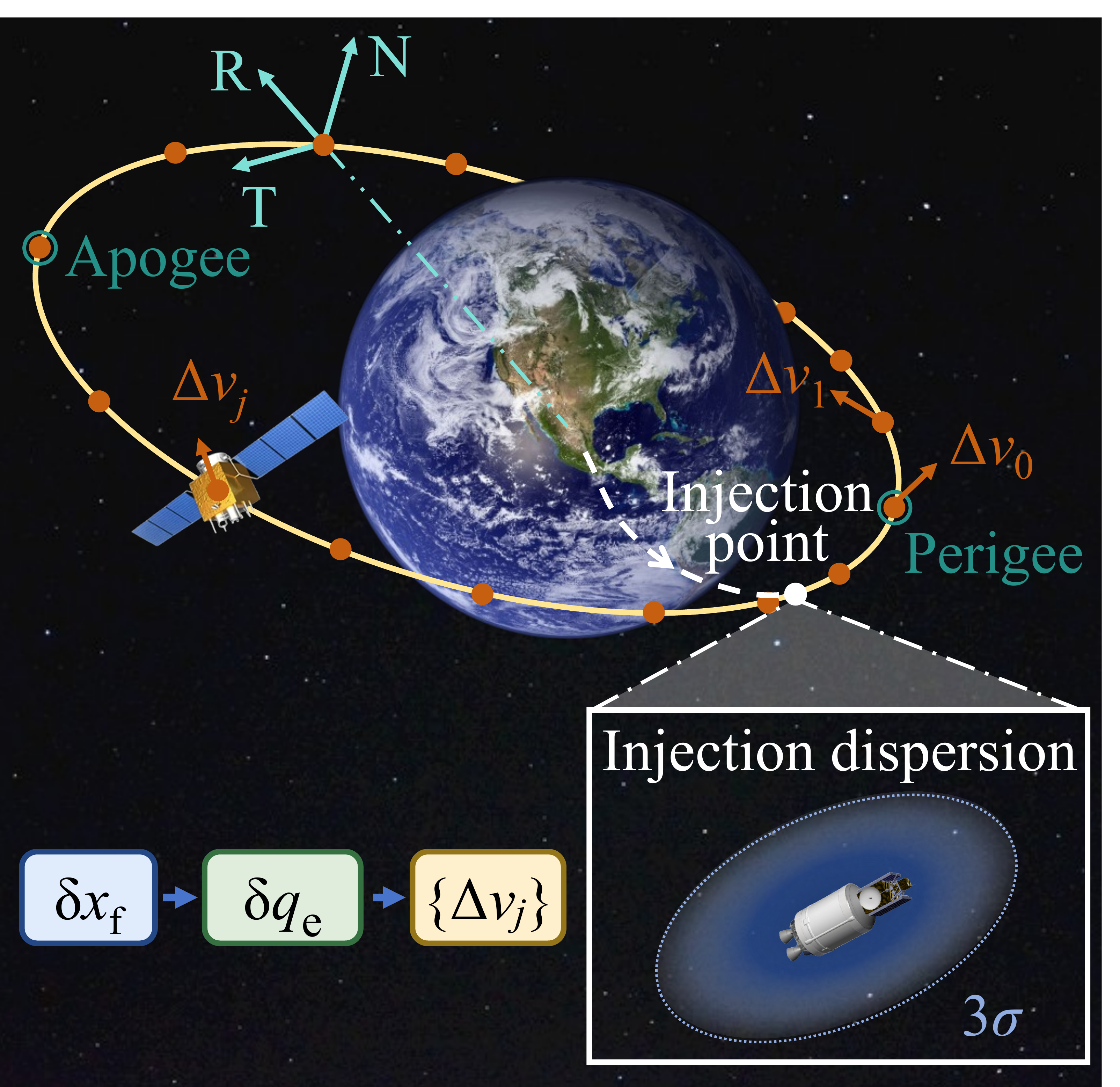}
\caption{Construction of the post-injection correction metric.}
\label{fig:mee_schematic}
\end{figure}

For a mission that does not prescribe the injection phase along the target orbit, use the phase-free modified equinoctial element (MEE) vector
$
  q_{\mathrm{e}}=
  \left[
  p_{\mathrm{e}}\ \ 
  f_{\mathrm{e}}\ \ 
  g_{\mathrm{e}}\ \ 
  h_{\mathrm{e}}\ \ 
  k_{\mathrm{e}}
  \right]^{\mathsf T}.
$
Here $p_{\mathrm{e}}$ is the semi-latus rectum, $f_{\mathrm{e}}$ and $g_{\mathrm{e}}$ are the eccentricity-vector components, and $h_{\mathrm{e}}$ and $k_{\mathrm{e}}$ describe the orbital plane. If the injection phase is prescribed, true longitude can be included as an additional terminal coordinate. Linearization about the nominal injection state gives
$
 \delta q_{\mathrm{e}}
 =G_{q_{\mathrm{e}}}\delta x_{\mathrm{f}},
$
where $G_{q_{\mathrm{e}}}$ is the Cartesian-to-MEE Jacobian and its mass column is zero.

To quantify the post-injection correction effort associated with $\delta q_{\mathrm{e}}$, consider a set of candidate correction locations with true longitudes $L_j$, $j=1,\ldots,N_L$.
At each candidate correction location, a local radial-transverse-normal (RTN) frame is defined, where the $R$ direction is radial, the $T$ direction is transverse in the orbital plane, and the $N$ direction is normal to the orbital plane, as illustrated in Fig.~\ref{fig:mee_schematic}. The phase-free Gauss variational mapping from the RTN acceleration $a_{\mathrm{RTN}}=\col(a_{\mathrm{R}},a_{\mathrm{T}},a_{\mathrm{N}})$ to the MEE rate is
\begin{equation}
 \dot q_{\mathrm{e}}
 =B_{q_{\mathrm{e}}}(q_{\mathrm{e}},L)a_{\mathrm{RTN}},
 \label{eq:gauss_mee}
\end{equation}
where
\begin{equation}
 B_{q_{\mathrm{e}}}(q_{\mathrm{e}},L)
 =\sqrt{\frac{p_{\mathrm{e}}}{\mu}}
 \begin{bmatrix}
  0&2p_{\mathrm{e}}/w_L&0\\
  \sin L&((w_L+1)\cos L+f_{\mathrm{e}})/w_L&-g_{\mathrm{e}}\chi_L/w_L\\
  -\cos L&((w_L+1)\sin L+g_{\mathrm{e}})/w_L&f_{\mathrm{e}}\chi_L/w_L\\
  0&0&s_L^2\cos L/(2w_L)\\
  0&0&s_L^2\sin L/(2w_L)
 \end{bmatrix}.
 \label{eq:gauss_mee_matrix}
\end{equation}
and $w_L=1+f_{\mathrm{e}}\cos L+g_{\mathrm{e}}\sin L$, $s_L^2=1+h_{\mathrm{e}}^2+k_{\mathrm{e}}^2$, $\chi_L=h_{\mathrm{e}}\sin L-k_{\mathrm{e}}\cos L$.
A small impulse $\Delta v_j\in\R^3$ at $L_j$ produces
$B_{q_{\mathrm{e}}}(q_{\mathrm{e}}^\star,L_j)\Delta v_j$. Stacking the candidate correction impulses gives
$
  H_{\mathrm{orb}}
  =
  \left[
  B_{q_{\mathrm{e}}}(q_{\mathrm{e}}^\star,L_1)
  \quad\cdots\quad
  B_{q_{\mathrm{e}}}(q_{\mathrm{e}}^\star,L_{N_L})
  \right]
$ and 
$\Delta V=\col(\Delta v_1,\ldots,\Delta v_{N_L}).$
Thus the equivalent post-injection correction velocity associated with a terminal orbital error is defined as the minimum total impulse magnitude required to eliminate the corresponding orbital deviation:
\begin{subequations}\label{prob:minimum_impulse}
\begin{align}
 \mathcal P_{3\mathrm{a}}:\qquad
 J_{\mathrm{corr},1}^\star
 =\min_{\{\Delta v_j\}}\quad&
 \sum_{j=1}^{N_L}\|\Delta v_j\|_2\\
 \mathrm{s.t.}\quad&
 H_{\mathrm{orb}}\Delta V
 =-G_{q_{\mathrm{e}}}\delta x_{\mathrm{f}}.
\end{align}
\end{subequations}
Because $J_{\mathrm{corr},1}^\star$ has units of velocity, it provides an equivalent $\Delta v$ measure of injection error. Physically, it represents the minimum post-injection velocity increment required, within the selected correction locations and the local orbital model, to remove the terminal orbital deviation.

The number and longitudes of the candidate correction locations can be selected according to mission requirements. Since $H_{\mathrm{orb}}\in\mathbb R^{5\times 3N_L}$, at least two candidate locations are required for full row rank. However, if their true longitudes differ by $\pi$, as for perigee and apogee, the columns associated with the $N$-direction impulses become linearly dependent. In this case, at least one additional non-antipodal candidate location is required.

\subsection{Task-Oriented Terminal Weight}\label{sec:terminal_weight}

The preceding subsection provides a physical metric for evaluating
injection accuracy. The feedback synthesis additionally requires a
quadratic terminal weight, which is obtained from the least-squares
counterpart of that correction problem:
\begin{subequations}\label{prob:minimum_squared_impulse}
\begin{align}
 \mathcal P_{3\mathrm{b}}:\qquad
 J_{\mathrm{corr},2}^\star
 =\min_{\{\Delta v_j\}}\quad&
 \sum_{j=1}^{N_L}\|\Delta v_j\|_2^2\\
 \mathrm{s.t.}\quad&
 H_{\mathrm{orb}}\Delta V
 =-G_{q_{\mathrm{e}}}\delta x_{\mathrm{f}}.
\end{align}
\end{subequations}

\begin{theorem}[Orbit-Correction Terminal Weight]\label{thm:task_weight}
Assume $H_{\mathrm{orb}}$ has full row rank and define
\begin{equation}
  Q_{\mathrm{f},x}^{\mathrm{orb}}
  =G_{q_{\mathrm{e}}}^{\mathsf T}
  (H_{\mathrm{orb}}H_{\mathrm{orb}}^{\mathsf T})^{-1}
  G_{q_{\mathrm{e}}}.
  \label{eq:orbital_terminal_weight}
\end{equation}
Then every sufficiently small terminal perturbation $\delta x_{\mathrm{f}}$ satisfies
\begin{equation}
 \delta x_{\mathrm{f}}^{\mathsf T}
 Q_{\mathrm{f},x}^{\mathrm{orb}}\delta x_{\mathrm{f}}
 =J_{\mathrm{corr},2}^\star.
 \label{eq:terminal_energy_identity}
\end{equation}
\end{theorem}

\begin{proof}
The Lagrangian of $\mathcal P_{3\mathrm{b}}$ is
$\mathcal L=\Delta V^{\mathsf T}\Delta V
+2\lambda^{\mathsf T}(H_{\mathrm{orb}}\Delta V
+G_{q_{\mathrm{e}}}\delta x_{\mathrm{f}})$.
Stationarity with respect to $\Delta V$ gives
$\partial\mathcal L/\partial\Delta V
=2\Delta V+2H_{\mathrm{orb}}^{\mathsf T}\lambda=0$,
so $\Delta V=-H_{\mathrm{orb}}^{\mathsf T}\lambda$.
Substitution into the equality constraint yields
$-H_{\mathrm{orb}}H_{\mathrm{orb}}^{\mathsf T}\lambda
=-G_{q_{\mathrm{e}}}\delta x_{\mathrm{f}}$.
Because $H_{\mathrm{orb}}$ has full row rank,
$H_{\mathrm{orb}}H_{\mathrm{orb}}^{\mathsf T}$ is nonsingular, giving
\begin{equation}
   \lambda^\star=(H_{\mathrm{orb}}H_{\mathrm{orb}}^{\mathsf T})^{-1}
    G_{q_{\mathrm{e}}}\delta x_{\mathrm{f}},\qquad
    \Delta V^\star=-H_{\mathrm{orb}}^{\mathsf T}
    (H_{\mathrm{orb}}H_{\mathrm{orb}}^{\mathsf T})^{-1}
    G_{q_{\mathrm{e}}}\delta x_{\mathrm{f}}. 
\end{equation}
Therefore,
$J_{\mathrm{corr},2}^\star
=(\Delta V^\star)^{\mathsf T}\Delta V^\star
=\delta x_{\mathrm{f}}^{\mathsf T}
G_{q_{\mathrm{e}}}^{\mathsf T}
(H_{\mathrm{orb}}H_{\mathrm{orb}}^{\mathsf T})^{-1}
G_{q_{\mathrm{e}}}\delta x_{\mathrm{f}}
=\delta x_{\mathrm{f}}^{\mathsf T}
Q_{\mathrm{f},x}^{\mathrm{orb}}\delta x_{\mathrm{f}}$,
which proves Eq.~\eqref{eq:terminal_energy_identity}.
\end{proof}

For a zero-mean terminal error with covariance $P_{\mathrm{f}}$,
Eq.~\eqref{eq:terminal_energy_identity} gives
\begin{equation}
\E[J_{\mathrm{corr},2}^\star]
= \E[\delta x_{\mathrm{f}}^{\mathsf T}
 Q_{\mathrm{f},x}^{\mathrm{orb}}\delta x_{\mathrm{f}}]
=\tr(Q_{\mathrm{f},x}^{\mathrm{orb}}P_{\mathrm{f}}).
\label{eq:expected_orbit_energy}
\end{equation}
Hence the quadratic terminal penalty associated with
$Q_{\mathrm{f},x}^{\mathrm{orb}}$ has the direct interpretation of
expected quadratic post-injection correction energy. By choosing a
positive scaling factor $\lambda_{\mathrm{f}}$ to set the overall
terminal penalty magnitude, the terminal matrix in Eq.~\eqref{eq:continuous_cost} is therefore chosen as
\begin{equation}
Q_{\mathrm{f}}
=\lambda_{\mathrm{f}}Q_{\mathrm{f},x}^{\mathrm{orb}},
\qquad
\delta x_{\mathrm{f}}^{\mathsf T}Q_{\mathrm{f}}\delta x_{\mathrm{f}}
=\lambda_{\mathrm{f}}J_{\mathrm{corr},2}^\star,
\qquad
\lambda_{\mathrm{f}}>0.
\label{eq:weighted_terminal}
\end{equation}

\subsection{Alternating Stochastic Replanning}\label{sec:alternating}

\begin{figure}[t]
\centering
\includegraphics[width=0.8\columnwidth]{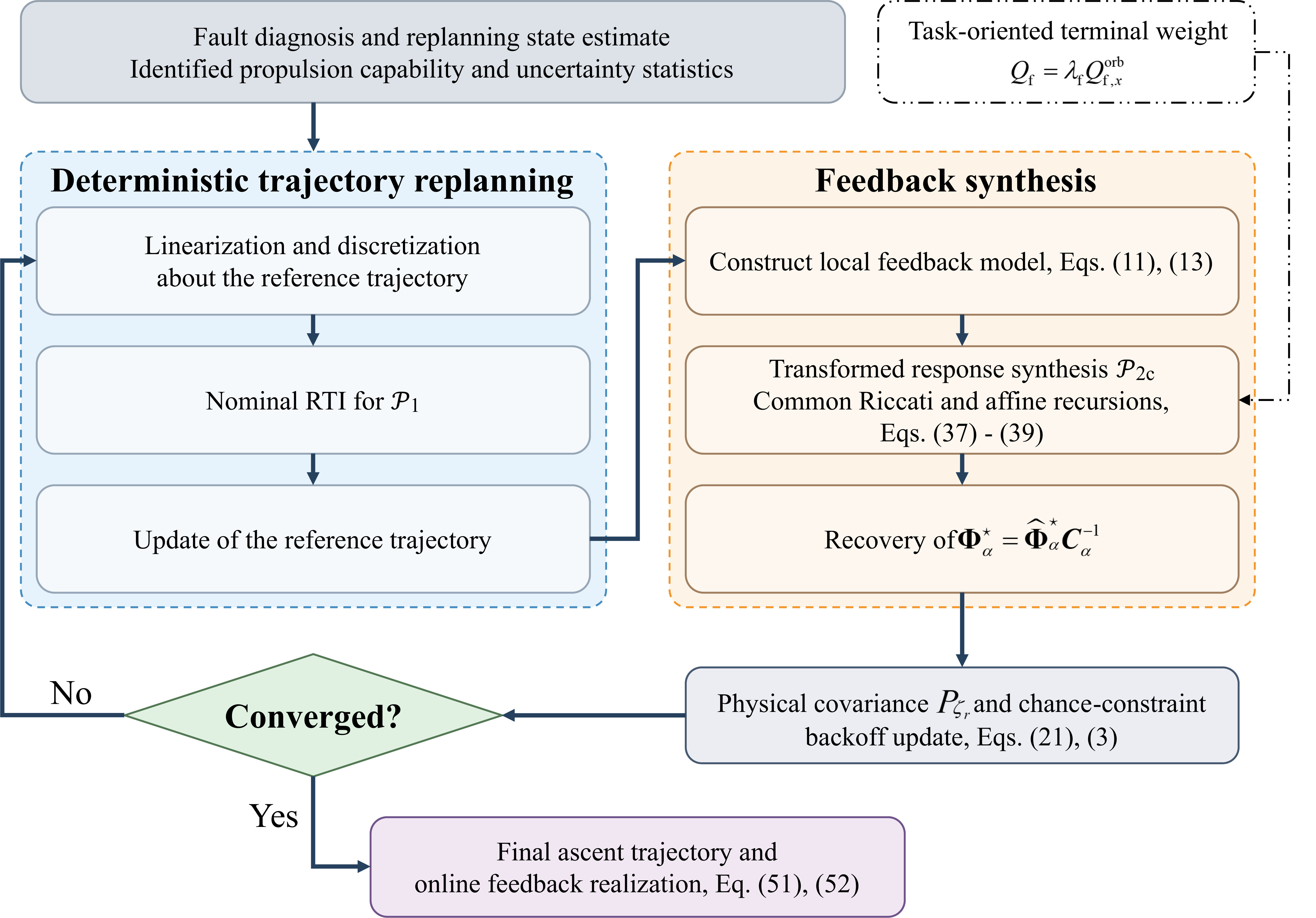}
\caption{Alternating stochastic ascent replanning procedure.}
\label{fig:architecture}
\end{figure}

With the task-oriented terminal weight from Sec.~\ref{sec:terminal_weight}, the feedback synthesis in Sec.~\ref{sec:sls} is coupled with the nominal replanning problem in Sec.~\ref{sec:problem} through the chance-constraint backoffs, forming the alternating stochastic replanning procedure shown in Fig.~\ref{fig:architecture}. At each outer iteration, the current backoffs are first imposed on the nominal problem $\mathcal P_1$, and the reference trajectory is updated by one real-time iteration (RTI), defined here as solving a single warm-started convex SCP subproblem rather than iterating to full deterministic convergence. Based on the updated reference trajectory, the local perturbation model, feedback cost, and covariance structure are reconstructed, and the task-oriented terminal weight is incorporated into the transformed feedback problem $\mathcal P_{2\mathrm{c}}$. The latter is then solved using the recursions in Sec.~\ref{sec:sls}, after which the physical system responses are recovered according to Theorem~\ref{thm:structured}. These responses are used to evaluate the constraint covariances $P_{\zeta_r}$ and update the chance-constraint backoffs through Eq.~\eqref{eq:chance_backoff}. If the prescribed convergence criteria are not satisfied, the updated backoffs are returned to $\mathcal P_1$ for the next outer iteration. Otherwise, the resulting reference trajectory and system responses are retained for online feedback realization in Sec.~\ref{sec:online_feedback}. The stochastic replanning procedure is therefore solved only once after fault diagnosis, while the subsequent ascent uses the resulting feedback responses without repeated online replanning. In this manner, the nominal trajectory and feedback policy are successively adapted to each other without solving a single large-scale coupled optimization problem.

The alternating iteration terminates when both the scaled change in the nominal trajectory and the normalized change in the chance-constraint backoffs satisfy their prescribed tolerances, or when the maximum number of outer iterations $N_{\mathrm{max}}$ is reached. Related local convergence and stationary-point analyses for alternating backoff updates are given in Ref.~\cite{Messerer2026RiccatiZORO}.

\subsection{Online Feedback Realization}\label{sec:online_feedback}

Although the causal feedback matrix can be recovered from the optimized system responses through Eq.~\eqref{eq:feedback_response_relation}, this requires the additional inversion of $\bm\Phi^x$ and therefore introduces unnecessary computation during online execution. We instead use the algebraically equivalent innovation realization
\begin{equation}
    \nu_k
    =\sum_{j=-1}^{k-1}\Phi^u_{k,j}\widehat d_j,
    \label{eq:innovation_control}
\end{equation}
where $\widehat d_{-1}=\chi_0$ and, after $\chi_{j+1}$ is measured, the disturbance innovation is reconstructed as $\widehat d_j=\chi_{j+1}-\widetilde A_j\chi_j-\widetilde B_j\nu_j$. Thus, the optimized system responses can be used directly without explicitly recovering the feedback gains. Since $\nu_k=\delta u_{k+1}-\delta u_k$, the corresponding corrected nodal control angles satisfy $u_k^{\mathrm c}=\bar u_k+\delta u_k$ and $u_{k+1}^{\mathrm c}=\bar u_{k+1}+\delta u_k+\nu_k$. Note that $u$ denotes the thrust-direction control angles rather than the thrust magnitude, so this implementation does not require high-accuracy online identification of the actual thrust level.

The corrected nodal commands are scheduled online according to the specific mechanical energy introduced in Sec.~\ref{sec:foh_feedback}, which can be evaluated directly from the navigation solution.  Over interval $k$, let $\bar\varepsilon_k(\sigma)=\varepsilon(\bar x(t_k+\sigma))$, $\sigma\in[0,\tau_k]$, denote the monotonic nominal energy profile. Its inverse defines the interpolation coordinate $s_k^\varepsilon(x)=\bar\varepsilon_k^{-1}(\varepsilon(x))/\tau_k$, and the commanded control angles are
\begin{equation}
    u^{\mathrm c}(x)
    =(1-s_k^\varepsilon(x))u_k^{\mathrm c}
    +s_k^\varepsilon(x)u_{k+1}^{\mathrm c}.
    \label{eq:energy_control_interpolation}
\end{equation}
Along the nominal trajectory, this mapping satisfies $s_k^\varepsilon(\bar x(t))=(t-t_k)/\tau_k$ and therefore recovers the nominal time-based FOH control schedule, while under perturbations its timing adapts to the measured energy progress. In the numerical implementation, the inverse energy--time mapping is tabulated along each nominal interval.

\section{Numerical Results}\label{sec:numerical}

The numerical study consists of three groups of experiments. First, a representative post-fault case is used to examine the closed-loop behavior and the accuracy of the linear covariance prediction. The same case is then used for terminal-weight comparison, correction metric validation, and uncertainty source analysis. Second, the proposed method is compared with classical ascent guidance methods under representative fault conditions to examine terminal accuracy and constraint satisfaction. Finally, uncertainty sweeps and a broader fault grid are used to evaluate the sensitivity, recovery capability, and computational performance of the complete stochastic replanning method.

\begin{table}[!htbp]
\caption{Common numerical parameters.}
\label{tab:setup}
\centering\footnotesize
\setlength{\tabcolsep}{2pt}
\renewcommand{\arraystretch}{1.55}
\begin{tabular}{
@{}
>{\raggedright\arraybackslash}p{0.05\textwidth}
>{\centering\arraybackslash}p{0.13\textwidth}
>{\raggedright\arraybackslash}p{0.05\textwidth}
>{\centering\arraybackslash}p{0.13\textwidth}
>{\raggedright\arraybackslash}p{0.05\textwidth}
>{\centering\arraybackslash}p{0.13\textwidth}
>{\raggedright\arraybackslash}p{0.06\textwidth}
>{\centering\arraybackslash}p{0.14\textwidth}
@{}}
\toprule
Parameter & Value & Parameter & Value & Parameter & Value & Parameter & Value \\
\midrule
$m_{\min,1}$ & $1.5\times10^5\ \mathrm{kg}$
& $\bar C_{\mathrm D}^{(1,2)}$ & $0.30$
& $3\sigma_{\eta,\mathrm T}^{(1)}$ & $0.60\%$
& $\varrho_{\eta,T\dot m}^{(1,2)}$ & $1$ \\

$m_{\min,2}$ & $1.6\times10^4\ \mathrm{kg}$
& $q_{\mathrm d,\max}$ & $80\ \mathrm{kPa}$
& $3\sigma_{\eta,\dot m}^{(1)}$ & $0.60\%$
& $n_\chi,\ n_\nu$ & $10,\ 2$ \\

$m_{\mathrm{prop},1}$ & $3.9\times10^5\ \mathrm{kg}$
& $q_{\mathrm h,\max}$ & $50\ \mathrm{kW/m^2}$
& $3\sigma_{b,\mathrm T}^{(2)}$ & $0.60\%$
& $n_{b}$ & $4$ \\

$m_{\mathrm{prop},2}$ & $1.1\times10^5\ \mathrm{kg}$
& $\dot u_{\max}$ & $1.0\ \mathrm{deg/s}$
& $3\sigma_{b,\dot m}^{(2)}$ & $0.60\%$
& $\lambda_{\mathrm f}$ & $1.0\times10^4$ \\

$T_{1,\mathrm{vac}}^{\mathrm{nom}}$ & $7.9\ \mathrm{MN}$
& $N$ & $38$
& $3\sigma_{\eta,\mathrm T}^{(2)}$ & $0.30\%$
& $\alpha$ & $1.0\times10^{-6}s_{\rm reg}$ \\

$T_{2,\mathrm{vac}}^{\mathrm{nom}}$ & $0.98\ \mathrm{MN}$
& $3\sigma_{r,0}$ & $60\ \mathrm{m}$
& $3\sigma_{\eta,\dot m}^{(2)}$ & $0.30\%$
& $N_L$ & $4$ \\

$I_{\mathrm{sp},1}^{\mathrm{nom}}$ & $3.0\times10^2\ \mathrm{s}$
& $3\sigma_{v,0}$ & $1.5\ \mathrm{m/s}$
& $3\sigma_{b,C_{\mathrm D}}^{(1,2)}$ & $30\%$
& $N_{\max}$ & $12$ \\

$I_{\mathrm{sp},2}^{\mathrm{nom}}$ & $3.5\times10^2\ \mathrm{s}$
& $3\sigma_{m,0}$ & $60\ \mathrm{kg}$
& $3\sigma_{\eta,C_{\mathrm D}}^{(1,2)}$ & $0$
& $N_{\mathrm{MC}}$ & $1000$ \\

$A_{\mathrm e,1}$ & $5.9\ \mathrm{m^2}$
& $\pi_r$ & $5.0\times10^{-3}$
& $3\sigma_{b,\rho}^{(1,2)}$ & $24\%$
& $R_{\mathrm n}$ & $1.0\ \mathrm{m}$ \\

$A_{\mathrm e,2}$ & $0\ \mathrm{m^2}$
& $3\sigma_{b,\mathrm T}^{(1)}$ & $2.1\%$
& $3\sigma_{\eta,\rho}^{(1,2)}$ & $9.0\%$
& $Q_{\mathrm c}$ & $\diag(10^{-3}I_3,10^2I_3,0)$ \\

$S$ & $11\ \mathrm{m^2}$
& $3\sigma_{b,\dot m}^{(1)}$ & $2.1\%$
& $\varrho_{b,T\dot m}^{(1,2)}$ & $1$
& $R_{\mathrm c}$ & $2.3\times10^2 I_2$ \\
$L_{\mathrm{imp}}$ & $4.7\times10^2\ \mathrm{km}$ & $W_{\mathrm{imp}}$ & $26\ \mathrm{km}$ & & & & \\
\bottomrule
\end{tabular}
\end{table}

The common vehicle, environmental, uncertainty, and algorithm parameters are summarized in Table~\ref{tab:setup}. Standard terrestrial constants such as $\mu$, $R_{\mathrm E}$, ${\omega}_{\mathrm E}$, and $g_0$ take their conventional values and are not repeated in the table. The launch vehicle model is based on the Falcon~9 Full Thrust configuration and consists of two powered phases, with propulsion faults considered only during the first stage. Let $m_{\min,i}$ and $m_{\mathrm{prop},i}$ denote the minimum mass and propellant mass associated with phase $i$, respectively, and let $T_{i,\mathrm{vac}}^{\mathrm{nom}}$ and $I_{\mathrm{sp},i}^{\mathrm{nom}}$ denote the corresponding pre-fault nominal vacuum thrust and specific impulse. The mission targets an elliptical orbit with perigee and apogee altitudes of $300$ and $500~\mathrm{km}$, respectively. The initial-state covariance is parameterized as $P_0=\operatorname{diag}(\sigma_{r,0}^2 I_3,\sigma_{v,0}^2 I_3,\sigma_{m,0}^2)$. Here, $N_L$ is the number of candidate post-injection correction locations and $N_{\mathrm{MC}}$ is the number of nonlinear Monte Carlo (MC) realizations used for each reported statistical result, unless otherwise specified. Four candidate locations are used to construct the post-injection correction metric and the associated terminal weight, with their true longitudes uniformly distributed over one orbital revolution starting from the reference perigee. The regularization is set to $\alpha=10^{-6}s_{\mathrm{reg}}$, where $s_{\mathrm{reg}}$ is defined as the median of the positive eigenvalues of the diagonal blocks of $\bm M_{\mathrm{phys}}$. The impact region is defined as a downrange-aligned rectangle of length $L_{\mathrm{imp}}$ and width $W_{\mathrm{imp}}$. Throughout Sec. V, $\Delta v$ denotes the equivalent post-injection correction velocity $J_{\mathrm{corr},1}^{*}$, evaluated for each nonlinear Monte Carlo realization using the metric defined in Sec. IV.A and P95 denotes the 95th percentile over the Monte Carlo samples.

The implementation is primarily in Python, and the convex subproblems are solved using Clarabel. All timing tests are performed on a desktop with a CPU clocked at up to 5.7 GHz and 64 GB RAM using a single CPU thread. Selected performance-critical loops are implemented as compiled C kernels to reduce Python overhead.

\subsection{Terminal Weight Selection and Correction Metric Validation}\label{sec:terminal_weight_selection}

A representative fault occurring at $t=30~\mathrm{s}$ is considered first, with a $5\%$ loss of thrust and a $5\%$ loss of mean specific impulse. Figure~\ref{fig:representative_trajectory} shows the representative closed-loop trajectories obtained with the proposed task-oriented terminal weight. The terminal Monte Carlo
dispersion agrees well with the corresponding linear covariance
prediction and is largely contained within the predicted $3\sigma$
bounds. The relative Frobenius difference between the empirical and
predicted terminal state covariances is $3.67\%$, further
supporting the accuracy of the linear covariance approximation at
terminal injection.

\begin{figure}[!htbp]
\centering
\includegraphics[width=\linewidth]{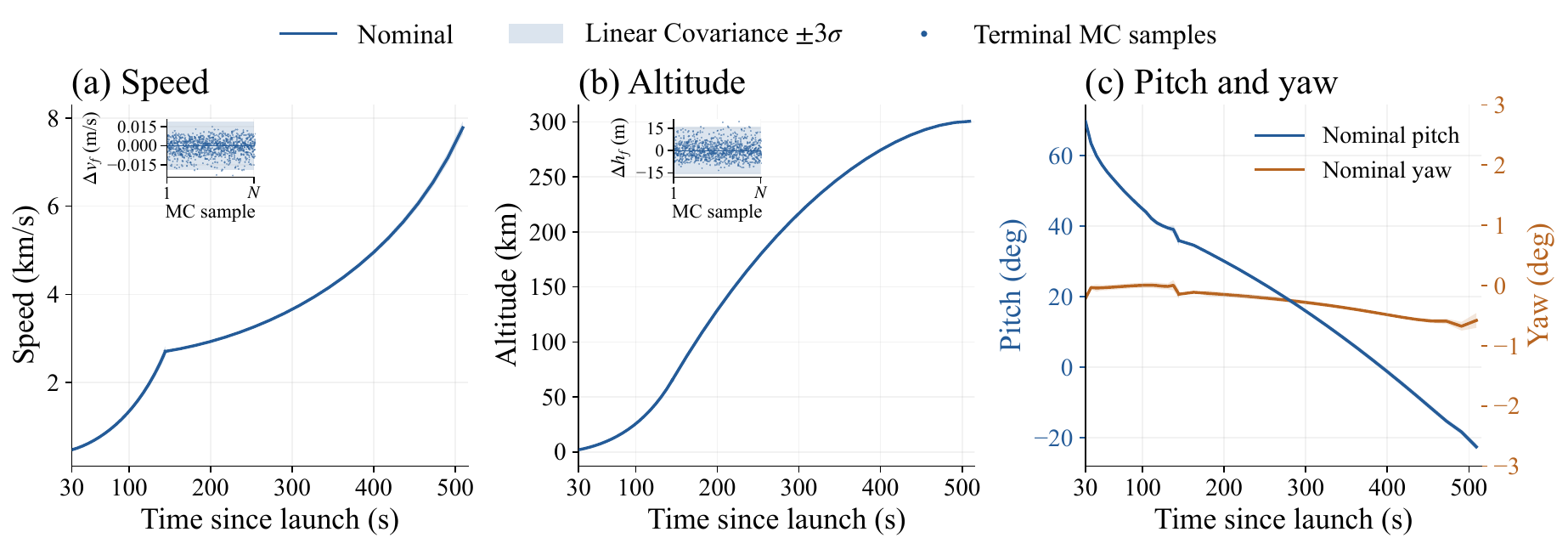}
\caption{Representative closed-loop trajectories and terminal dispersion.}
\label{fig:representative_trajectory}
\end{figure}

To assess the benefit of the task-oriented terminal weight, a trace-matched diagonal terminal weight is introduced for comparison. Specifically, the trace-matched diagonal terminal weight is chosen as $\operatorname{diag}(\gamma I_6,0)$, where $\gamma$ is selected such that its trace equals that of the proposed task-oriented weight. The nominal process weights are shown in Table~\ref{tab:setup} and are additionally varied by a common positive factor $c$, with $c=1$ corresponding to the nominal process weight configuration. The results are summarized in Table~\ref{tab:joint_weights}. The task-oriented design converges for all tested values of $c$ and yields mean equivalent post-injection correction velocities of only $0.00367$--$0.135~\mathrm{m/s}$. In contrast, the trace-matched diagonal weight gives substantially larger values of $0.754$--$1.94~\mathrm{m/s}$ for the accepted cases and fails to converge for several process weight settings. The diagonal weight penalizes terminal dispersion in directions that have little influence on the required post-injection correction and can therefore induce unnecessarily aggressive feedback action. This, in turn, can increase the covariance associated with intermediate path and control quantities, leading to excessive chance constraint tightening and eventually rendering the subsequent nominal replanning problem infeasible. Thus, matching only the overall magnitude of the terminal weight does not reproduce the performance obtained by shaping the terminal covariance according to the post-injection correction task.

\begin{table}[!htbp]
\caption{Comparison of task-oriented and trace-matched diagonal
terminal weights under process weight scaling.}
\label{tab:joint_weights}
\centering\footnotesize
\setlength{\tabcolsep}{4pt}
\begin{tabular}{rrrrrrr}
\toprule
$c$ 
& \shortstack[c]{Task outer\\iters.}
& \shortstack[c]{Task mean $\Delta v$\\(m/s)}
& \shortstack[c]{Task P95 $\Delta v$\\(m/s)}
& \shortstack[c]{Diag. outer\\iters.}
& \shortstack[c]{Diag. mean $\Delta v$\\(m/s)}
& \shortstack[c]{Diag. P95 $\Delta v$\\(m/s)} \\
\midrule
0.01 & 5 & 0.00370 & 0.00800 & Not converged & -- & -- \\
0.1 & 5 & 0.00367 & 0.00807 & Not converged & -- & -- \\
1 & 5 & 0.00381 & 0.00818 & Not converged & -- & -- \\
10 & 5 & 0.00377 & 0.00801 & Not converged & -- & -- \\
100 & 5 & 0.00543 & 0.0105 & Not converged & -- & -- \\
\num{1e+03} & 5 & 0.0241 & 0.0550 & 9 & 0.754 & 1.81 \\
\num{1e+04} & 5 & 0.135 & 0.311 & 6 & 1.94 & 4.50 \\
\bottomrule
\end{tabular}
\end{table}

To examine the contributions of different uncertainty sources to the terminal dispersion, an open-loop case is introduced in which the control sequence obtained from deterministic replanning is executed according to its planned time schedule without feedback. Because the physical uncertainty source groups are mutually independent, the terminal covariance can be decomposed as
$
P_{\mathrm f}=\sum_s P_{\mathrm f,s},
$
where $s$ indexes the uncertainty source groups listed in Table~\ref{tab:sources}. According to Eq.~\eqref{eq:expected_orbit_energy}, the expected quadratic post-injection correction energy is $\tr(Q_{\mathrm{f},x}^{\mathrm{orb}}P_{\mathrm f})$. The covariance decomposition therefore gives
$
\tr(Q_{\mathrm{f},x}^{\mathrm{orb}}P_{\mathrm f})
=
\sum_s \tr(Q_{\mathrm{f},x}^{\mathrm{orb}}P_{\mathrm f,s}),
$
which allows the contribution of each uncertainty source to the correction energy to be evaluated separately. As shown in Table~\ref{tab:sources}, the first-stage and upper-stage persistent propulsion biases account for $57.9\%$ and $40.8\%$ of the open-loop correction energy, respectively, and together contribute approximately $98.7\%$. This dominance shows that persistent propulsion biases are a major source of terminal injection dispersion and motivates their explicit treatment in the feedback synthesis. With the proposed method, their contributions are almost completely suppressed, and the remaining correction energy is dominated by the upper-stage interval-wise independent propulsion uncertainty, which accounts for approximately $99.9\%$. Since this uncertainty continues to enter near the end of powered flight, progressively less time remains for feedback to reduce its contribution to the terminal dispersion.

\begin{table}[!htbp]
\caption{Contributions of individual uncertainty sources to the expected quadratic post-injection correction energy.}
\label{tab:sources}
\centering\footnotesize
\begin{tabular}{lrrrr}
\toprule
\multicolumn{1}{c}{\multirow{2}{*}{Source}} & \multicolumn{2}{c}{Open loop} & \multicolumn{2}{c}{Proposed} \\
\cmidrule(lr){2-3}\cmidrule(lr){4-5}
& $\tr(Q_{\mathrm{f},x}^{\mathrm{orb}}P_{\mathrm f,s})$ & Share (\%) & $\tr(Q_{\mathrm{f},x}^{\mathrm{orb}}P_{\mathrm f,s})$ & Share (\%) \\
\midrule
First-stage propulsion bias & \num{799} & \num{57.9} & \num{1.24e-10} & \num{0.00201} \\
Upper-stage propulsion bias & \num{563} & \num{40.8} & \num{4.13e-09} & \num{0.0673} \\
First-stage independent & \num{3.84} & \num{0.279} & \num{2.54e-12} & \num{4.15e-05} \\
Upper-stage independent & \num{7.06} & \num{0.512} & \num{6.13e-06} & \num{99.9} \\
Density bias & \num{2.51} & \num{0.182} & \num{2.88e-14} & \num{4.70e-07} \\
Density independent & \num{0.0296} & \num{0.00215} & \num{1.86e-13} & \num{3.03e-06} \\
Drag-coefficient bias & \num{3.92} & \num{0.284} & \num{4.51e-14} & \num{7.35e-07} \\
Initial state & \num{0.356} & \num{0.0258} & \num{2.55e-12} & \num{4.16e-05} \\
\midrule
Total & \num{1.38e+03} & 100 & \num{6.13e-06} & 100 \\
\bottomrule
\multicolumn{5}{l}{\scriptsize Absolute contributions are reported in $\mathrm{m^2/s^2}$.}\\
\end{tabular}
\end{table}

The same open-loop Monte Carlo results are further used to validate the proposed post-injection correction metric. In the nonlinear orbital propagation, the correction maneuvers computed from the proposed metric are applied directly without reoptimization, and Table~\ref{tab:correction} reports the mean absolute orbital element errors before and after correction over the Monte Carlo runs. Here, $\Delta a$, $\Delta e$, $\Delta i$, $\Delta\Omega$, and $\Delta\omega$ denote the errors in semimajor axis, eccentricity, inclination, right ascension of the ascending node, and argument of perigee, respectively. The magnitudes of all five mean absolute orbital element errors are reduced by $99.55$--$99.95\%$ after correction. These nonlinear results verify that the proposed metric provides an accurate measure of the velocity correction required to remove the terminal orbital error.

\begin{table}[!htbp]
\caption{Nonlinear validation of the post-injection correction metric.}
\label{tab:correction}
\centering\footnotesize
\begin{tabular}{lrrr}
\toprule
Element & Before MAE & After MAE & Decrease (\%) \\
\midrule
$\Delta a$ (m) & \num{5.48e+04} & \num{213} & 99.6116 \\
$\Delta e$ & \num{0.00581} & \num{2.62e-05} & 99.5492 \\
$\Delta i$ (deg) & \num{0.00758} & \num{4.02e-06} & 99.9469 \\
$\Delta\Omega$ (deg) & \num{0.00626} & \num{4.13e-06} & 99.9340 \\
$\Delta\omega$ (deg) & \num{25.9} & \num{0.0573} & 99.7788 \\
\bottomrule
\multicolumn{4}{l}{\scriptsize MAE denotes the mean absolute error over the open-loop Monte Carlo samples.}\\
\end{tabular}
\end{table}

\subsection{Comparison with Conventional Ascent Guidance}

Two conventional guidance combinations are considered for comparison \cite{Jiang2026AdaptiveNotch}. The first employs first-stage programmed guidance followed by powered explicit guidance (PEG), whereas the second employs the perturbation guidance method (PGM) followed by the iterative guidance mode (IGM). Both combinations use the trajectory obtained from deterministic post-fault replanning as their reference. The programmed guidance and PGM implementations follow the representative formulations in Refs.~\cite{Hawkins2019CHANGO,Zhang2025PerturbationGuidance}, respectively, while PEG and IGM follow the implementations in Ref.~\cite{Song2015Comparison}. For both terminal guidance methods, the original shutdown condition is modified to use the prescribed terminal specific energy, ensuring a consistent comparison of terminal injection accuracy. The proposed method and the two conventional combinations are evaluated using the same vehicle model, uncertainty model, and post-fault initial reference trajectory, so that the comparison primarily reflects their different treatment of feedback and uncertainty. Two representative faults are considered. Fault F01 uses the same fault scenario as in Sec.~\ref{sec:terminal_weight_selection}, with a fault onset at 30 s, a 5\% thrust loss, and a 5\% specific impulse loss, and is used to examine the heat flux constraint. Fault F02 considers a fault onset at 75 s, a 15\% thrust loss, and no specific impulse loss, and is used to examine the impact constraint. The two faults lead to different post-fault trajectory profiles and bring different path constraints close to their active limits, thereby providing complementary tests of terminal accuracy and constraint satisfaction. The dynamic pressure and control-rate limits are sufficiently loose and remain inactive in all cases considered in this section. Each case is evaluated using $3000$ nonlinear Monte Carlo simulations.

\begin{table}[!htbp]
\caption{Post-injection correction metric and empirical constraint violations for the guidance comparison.}
\label{tab:classical}
\centering\footnotesize
\setlength{\tabcolsep}{4pt}
\begin{tabular}{llrrrrrr}
\toprule
Fault & Method & Mean $\Delta v$ (m/s) & P95 $\Delta v$ (m/s) & Max. $\Delta v$ (m/s) & Heat viol. (\%) & Impact viol. (\%) \\
\midrule
F01 & Program + PEG & 0.00336 & 0.00818 & 0.0367 & 60.3 & 0 \\
F01 & PGM + IGM & 0.00445 & 0.0105 & 0.129 & 60.2 & 0 \\
F01 & Proposed & 0.00375 & 0.00822 & 0.0155 & 0.800 & 0 \\
F02 & Program + PEG & 0.00368 & 0.00860 & 0.0517 & 0.100 & 49.9 \\
F02 & PGM + IGM & 0.00474 & 0.0122 & 0.214 & 0.367 & 52.4 \\
F02 & Proposed & 0.00378 & 0.00799 & 0.0141 & 0.0667 & 0.567 \\
\bottomrule
\end{tabular}
\end{table}

Table~\ref{tab:classical} compares the terminal injection performance and empirical constraint violations of the three guidance schemes. The quantity $\Delta v$ reported in the table denotes the equivalent post-injection correction velocity evaluated using the metric introduced in Sec.~\ref{sec:metric}. Under both faults, PEG and the proposed method achieve comparable injection accuracy, with very small values of $\Delta v$, whereas IGM produces somewhat larger terminal errors. A clearer difference appears in constraint satisfaction. Under faults F01 and F02, the conventional guidance combinations exhibit frequent violations of the heat flux and impact location constraints, respectively, whereas violations under the proposed method remain rare. 

\begin{figure}[!htbp]
\centering
\includegraphics[width=\linewidth]{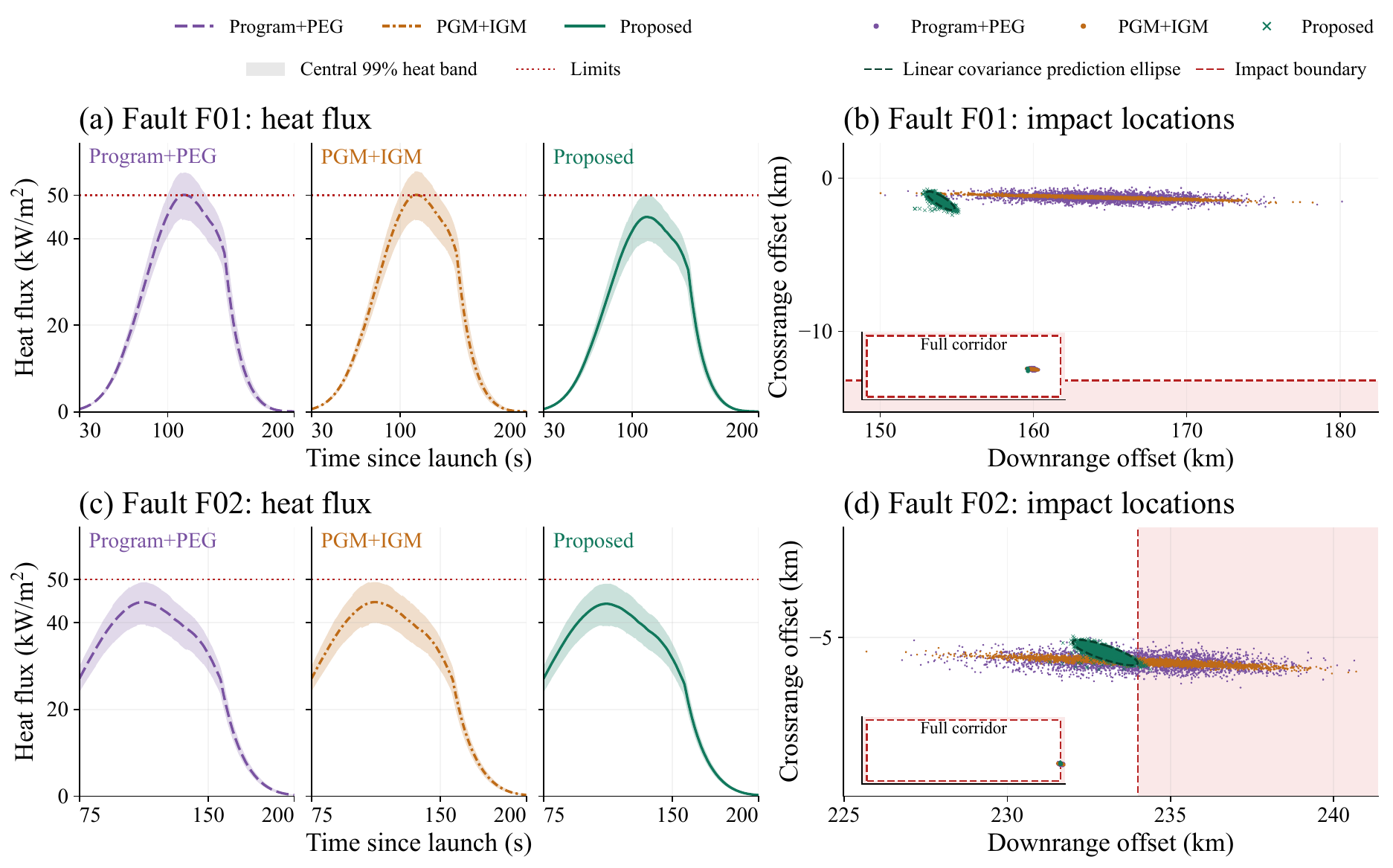}
\caption{Heat flux and impact location constraint responses for faults F01 and F02.}
\label{fig:classical}
\end{figure}

Figure~\ref{fig:classical} illustrates these two representative constraints. The left panels show the heat flux histories for faults F01 and F02, where the red dotted line denotes the heat-flux limit and the shaded bands represent the central $99\%$ Monte Carlo dispersion, consistent with the one-sided $99.5\%$ constraint satisfaction requirement. The right panels show the impact location dispersions for the two faults, together with the prescribed landing corridor. For the proposed method, the plotted covariance ellipse corresponds to the linear covariance prediction associated with the same one-sided $99.5\%$ chance-constraint level. Although the nominal trajectories generated by deterministic replanning satisfy the original limits, the conventional guidance schemes do not explicitly account for the dispersion induced by uncertainty. Consequently, a considerable fraction of the Monte Carlo trajectories cross the corresponding constraint boundaries. In contrast, the proposed stochastic replanning method propagates the uncertainty and introduces chance-constraint tightening during trajectory optimization. The resulting nominal trajectory therefore retains an uncertainty-dependent backoff from the original constraint boundary, substantially reducing the empirical violation rate while preserving terminal injection accuracy. In addition, the proposed feedback also yields a significantly smaller impact-location dispersion than the two conventional guidance schemes: the $1\sigma$ downrange dispersion is reduced by $81.6$--$90.3\%$ relative to Program~+~PEG and by $81.4$--$90.7\%$ relative to PGM~+~IGM.

\subsection{Uncertainty and Fault Envelopes}

\begin{figure}[!htbp]
\centering
\includegraphics[width=\linewidth]{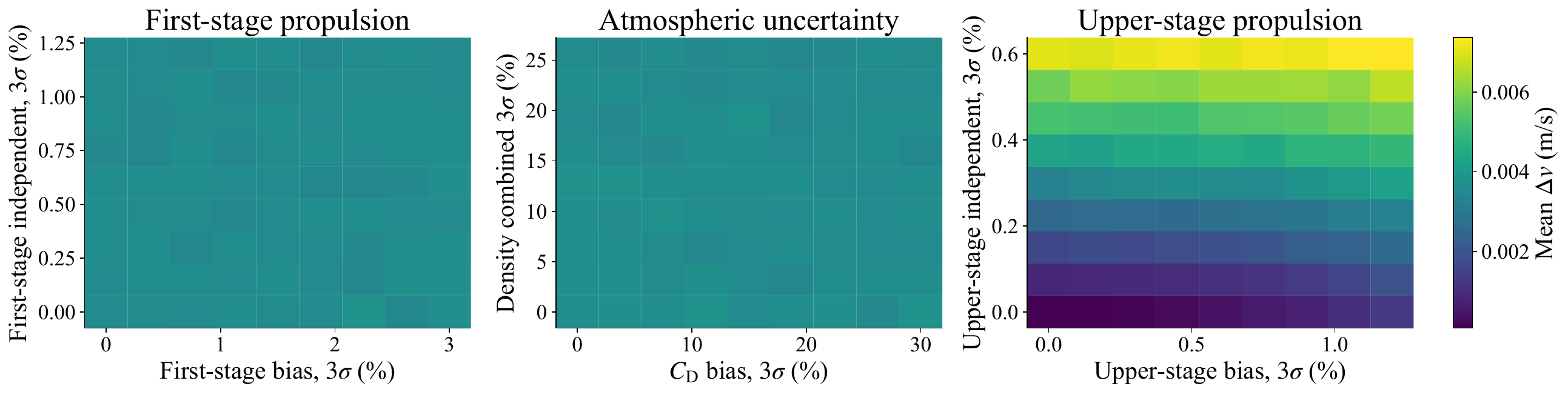}
\caption{Sensitivity of the mean equivalent post-injection correction velocity to uncertainty levels.}
\label{fig:uncertainty_sweep}
\end{figure}

The uncertainty sweeps consider a fault occurring at $45~\mathrm{s}$, with a $10\%$ loss of thrust and no loss of mean specific impulse. The sensitivity to uncertainty levels is first examined by varying two uncertainty components at a time while keeping all other uncertainty parameters at their nominal values in Table~\ref{tab:setup}. Figure~\ref{fig:uncertainty_sweep} shows three two-dimensional scans. The first varies the first-stage persistent and interval-wise independent thrust uncertainties over $3\sigma_{b,T}^{(1)}=0$--$3.0\%$ and $3\sigma_{\eta,T}^{(1)}=0$--$1.2\%$, respectively; the second varies the drag-coefficient and atmospheric-density uncertainties over $0$--$30\%$ and $0$--$25.6\%$; and the third similarly varies the upper-stage persistent and interval-wise independent thrust uncertainties over $0$--$1.2\%$ and $0$--$0.60\%$. The color in each cell represents the mean equivalent post-injection correction velocity $\Delta v$ obtained from the post-injection correction metric under the corresponding uncertainty levels. The first-stage propulsion and aerodynamic uncertainty scans show only minor changes in the resulting \(\Delta v\) over the tested ranges. The upper-stage persistent uncertainty likewise has a limited effect, whereas increasing the upper-stage interval-wise independent thrust uncertainty produces a much more pronounced degradation in injection accuracy. Even in this most sensitive case, however, the overall variation in $\Delta v$ remains limited, and high injection accuracy is maintained throughout the tested uncertainty ranges. This trend is consistent with the uncertainty-source decomposition in Table~\ref{tab:sources}: disturbances entering late in the upper-stage burn leave less remaining time for feedback to reduce their contribution to the terminal dispersion.

\begingroup
\begin{table}[!htbp]
\caption{Fault recovery performance and online replanning results.}\label{tab:faults}
\centering\small
\setlength{\tabcolsep}{3pt}
\begin{tabular}{lrrrrrrrr}
\toprule
Fault & $t_\mathrm{fault}$ (s) & Thrust loss (\%) & $I_{\rm sp}$ loss (\%) & Outer iters. & Replanning time (ms) & Mean $\Delta v$ (m/s) & Viol. (\%) \\
\midrule
F01 & 30 & 5 & 5 & 5 & 412 & 0.00371 & 1.10 \\
F02 & 75 & 15 & 0 & 5 & 286 & 0.00380 & 1.60 \\
F03 & 15 & 10 & 0 & 6 & 433 & 0.00369 & 1.20 \\
F04 & 15 & 20 & 0 & 6 & 389 & 0.00402 & 0.500 \\
F05 & 15 & 30 & 0 & 8 & 607 & 0.00471 & 0.500 \\
F06 & 45 & 10 & 0 & 4 & 270 & 0.00354 & 1.50 \\
F07 & 45 & 20 & 0 & 4 & 279 & 0.00373 & 0.500 \\
F08 & 45 & 30 & 0 & 5 & 402 & 0.00398 & 0.800 \\
F09 & 75 & 40 & 0 & 4 & 270 & 0.00399 & 0.500 \\
F10 & 75 & 40 & 5 & 4 & 262 & 0.00434 & 1.10 \\
F11 & 75 & 40 & 10 & 5 & 329 & 0.00456 & 0.500 \\
F12 & 30 & 20 & 0 & 5 & 289 & 0.00374 & 0.600 \\
F13 & 30 & 20 & 5 & 5 & 298 & 0.00426 & 1.30 \\
F14 & 30 & 20 & 10 & 5 & 300 & 0.00440 & 0.200 \\
\bottomrule
\end{tabular}
\end{table}
\endgroup

A broader fault scan is then performed to evaluate the fault recovery capability and online computational performance of the proposed method. The tested cases cover fault times $t_\mathrm{fault}$ of $15$--$75~\mathrm{s}$, thrust losses of $5$--$40\%$, and specific-impulse losses of $0$--$10\%$. Table~\ref{tab:faults} reports the terminal injection accuracy, empirical constraint violation rate, number of outer iterations, and online replanning time for each fault. A Monte Carlo run is counted as a violation if any of the prescribed path constraints is violated during the flight. For consistency, cases F01 and F02 are reevaluated using \(N_{\mathrm{MC}}=1000\), and their statistics therefore differ slightly from those reported in Sec. V.B. For all cases in which replanning converges, the proposed method maintains high injection accuracy and low empirical violation rates. Most cases require only $4$--$8$ outer iterations and are solved in $262$--$607~\mathrm{ms}$, demonstrating subsecond stochastic trajectory--guidance replanning over a broad range of propulsion faults and indicating strong potential for onboard real-time application. The average time of each outer iteration consists of approximately $21.98~\mathrm{ms}$ for the nominal trajectory update, $33.89~\mathrm{ms}$ for response construction and feedback synthesis, $4.92~\mathrm{ms}$ for covariance propagation and constraint-backoff updates, and $7.21~\mathrm{ms}$ for the remaining nonlinear checks and outer-loop operations.

\FloatBarrier

\section{Conclusions}

This paper developed a fast stochastic ascent replanning method for launch vehicles after nonfatal propulsion faults. The method addresses persistent parametric uncertainty caused by incomplete post-fault propulsion identification. Nominal trajectory replanning and stochastic feedback synthesis are coordinated through covariance propagation and chance-constraint tightening. A fast system level synthesis formulation is used to capture the cross-time effects of persistent parametric uncertainty, while a task-oriented terminal weight is derived from the post-injection correction demand. Innovation-based response realization and energy-based control interpolation are further used for efficient online implementation and reduced sensitivity to burn-time variation. Nonlinear Monte Carlo simulations demonstrate stable convergence, high injection accuracy, low constraint-violation rates, and subsecond solution times over a broad range of fault and uncertainty conditions. The main conclusions are as follows.
\begin{enumerate}[
    label=(\arabic*),
    labelindent=0pt,
    labelwidth=1.4em,
    labelsep=0.4em,
    leftmargin=1.8em,
    itemindent=0pt,
    itemsep=0.3em,
    topsep=0.3em,
    parsep=0pt
]
\item Persistent propulsion biases dominate the open-loop terminal dispersion, contributing 98.7\% of the expected quadratic post-injection correction energy. The proposed feedback effectively suppresses this contribution without high-accuracy online identification.
\item The task-oriented terminal weight improves both convergence behavior and injection accuracy over the trace-matched diagonal design. It converges across all tested process weight settings and achieves a mean equivalent post-injection correction velocity as low as approximately 0.004 m/s, whereas the best converged diagonal design remains above 0.75 m/s.
\item Compared with conventional ascent guidance methods, the proposed framework substantially reduces empirical path-constraint violations from approximately 55\% to 1\% while maintaining high injection accuracy. The solution time ranges from 262 to 607 ms across the successfully replanned fault cases.
\end{enumerate}

\section*{Funding Sources}
This work was supported by the Beijing Municipal Natural Science Foundation (L251014) and the National Natural Science Foundation of China (12525204).

\bibliography{references}

\end{document}